\documentclass[10pt]{article}

\usepackage{amsmath,amssymb,amsfonts,amsthm,mathtools}
\usepackage{graphicx,xcolor}
\usepackage{hyperref}
\hypersetup{
    colorlinks=true,
    linkcolor=blue,
    citecolor=darkgray
}

\usepackage[T1]{fontenc}
\usepackage[mathscr]{eucal}
\usepackage{fullpage}
\usepackage{thmtools}
\usepackage{enumerate}
\usepackage{color}
\usepackage{mathrsfs}
\usepackage{subcaption}
\usepackage{longtable} 
\usepackage{verbatim}

\theoremstyle{plain}
\newtheorem{thm}{Theorem} 

\newtheorem{lem}[thm]{Lemma}
\newtheorem{pr}[thm]{Proposition}
\newtheorem{cor}[thm]{Corollary}

\theoremstyle{remark}
\newtheorem*{unremark}{Remark}

\definecolor{darkgreen}{rgb}{0.0, 0.8, 0.08}

\allowdisplaybreaks

\def\E{\mathbb{E}}
\def\J{\mathcal{J}}
\def\MM{\mathcal{M}}
\def\P{\mathbb{P}}
\def\Z{\mathbb{Z}}
\def\R{\mathbb{R}}
\def\ee{\varepsilon}
\def\la{\lambda}
\def\xx{{\bf x}}
\def\yy{{\bf y}}
\def\XX{{\bf X}}
\def\Mt{\tilde{M}}
\def\mN{\mathcal{N}}
\def\Var{{\rm Var} \,}
\def\Cov{{\rm Cov} \,}
\def\Cox{\hfill \Box}
\def\disp{\displaystyle}
\def\one{{\bf 1}}
\def\dilog{{\rm dilog}\,}
\def\region{{\mathcal R}}
\def\regiona{{\mathcal R}_1}
\def\regionb{{\mathcal R}_2}
\def\regionc{{\mathcal R}_3}
\def\phihat{{\ov{\phi}}}
\def\Sbar{{\overline{S}}}
\def\Tbar{{\overline{T}}}

\newcommand\ov[1]{\widehat{#1}}

\date{January 16, 2019}

\title{Counting partitions inside a rectangle}

\author{Stephen Melczer$^\dagger$, Greta Panova$^{\dagger,*}$, and Robin Pemantle$^\dagger$}

\begin{document}

\begin{titlepage}

\maketitle 

\abstract{We consider the number of partitions of $n$ whose Young diagrams fit inside an $m \times \ell$ rectangle; equivalently, we study the coefficients of the $q$-binomial coefficient $\binom{m+\ell}{m}_q$.  We obtain sharp asymptotics throughout the regime $\ell = \Theta (m)$ and $n = \Theta (m^2)$. Previously, sharp asymptotics were derived by Tak{\'a}cs~\cite{Takacs1986} only in the regime where $|n - \ell m /2| = O(\sqrt{\ell m (\ell + m)})$ using a local central limit theorem.  Our approach is to solve a related large deviation problem: we describe the tilted measure that produces configurations whose bounding rectangle has the given aspect ratio and is filled to the given proportion. Our results are sufficiently sharp to yield the first asymptotic estimates on the consecutive differences of these numbers when $n$ is increased by one and $m, \ell$ remain the same, hence significantly refining Sylvester's unimodality theorem and giving effective asymptotic estimates for related Kronecker and plethysm coefficients from representation theory.} 
\vfill
 


\vspace{0.4in}

{\sc 
\noindent $\dagger$University of Pennsylvania \\
Department of Mathematics \\
209 South 33rd Street \\
Philadelphia, PA 19104 \\ 
}
{\tt \{smelczer,pemantle\}@math.upenn.edu} \vspace{0.2in}

{\sc 
\noindent $*$University of Southern California \\
Los Angeles, CA 90089 \\
}
{\tt gpanova@usc.edu} \vspace{0.2in}

\noindent{Melczer} partially supported by an NSERC postdoctoral fellowship and NSF grant DMS-1612674\\
Panova partially supported by NSF grant DMS-1500423, on leave 2018--2019 at the University of Southern California.\\
Pemantle partially supported by NSF grant DMS-1612674.

\noindent

\end{titlepage}

\section{Introduction}

A partition $\la$ of $n$ is a sequence of weakly decreasing nonnegative integers $\la=(\la_1\geq \la_2\geq \ldots)$ whose sum $|\la|=\la_1+\la_2+\cdots$ is equal to $n$. The study of integer partitions is a classic subject with applications ranging from number theory to representation theory and combinatorics, and integer partitions with various restrictions on properties, such as part sizes or number of parts, occupy the field of partition theory~\cite{Andrews1976}.  The generating functions of integer partitions play a role in number theory and the theory of modular forms. In representation theory, integer partitions index the conjugacy classes and irreducible representations of the symmetric group $S_n$; they are also the signatures of the irreducible polynomial representation of $GL_n$ and give a basis for the ring of symmetric functions.  More recently, partitions have appeared in the study of interacting particle systems and other statistical mechanics models.

The number of partitions of $n$, typically denoted by $p(n)$ but here unconventionally\footnote{We use the notation $N_n$ to distinguish scenarios of probability with those of enumeration, both of which occur in the present manuscript.} by $N_n$, was implicitly determined by Euler via the generating function 
\[ \sum_{n=0}^{\infty} N_n q^n = \prod_{i=1}^\infty \frac{1}{1-q^i} \, . \]
There is no exact explicit formula for the numbers $N_n$.  The asymptotic formula
\begin{align}
\label{eq:hardy-ram}
N_n:=\#\{ \lambda \vdash n \} \sim \frac{1}{4n \sqrt{3}} \exp\left( \pi \sqrt{\frac{2n}{3}} \right)\,,
\end{align}
obtained by Hardy and Ramanujan~\cite{HardyRamanujan1918}, is considered to be the beginning of the use of complex variable methods for asymptotic enumeration of partitions (the so-called circle method).

Our goal is to obtain asymptotic formulas similar to~\eqref{eq:hardy-ram} for the number of partitions $\la$ of $n$ whose Young diagram fits inside an $m \times \ell$ rectangle, denoted 
\[ N_n(\ell,m):=\# \{ \la \vdash n: \la_1 \, \leq \, \ell, \quad \text{length}(\la) \, \leq \, m\}\, .\] 
These numbers are also the coefficients in the expansion of the $q$-binomial coefficient 
\[
\binom{\ell +m}{m}_q = \frac{ \prod_{i=1}^{\ell+m} (1-q^i)} {\prod_{i=1}^\ell (1-q^i) \prod_{i=1}^m (1-q^i)} = \sum_{n=0}^{\ell m} N_n(\ell,m)q^n \, . 
\]

The $q$-binomial coefficients are themselves central to enumerative and algebraic combinatorics. They are the generating functions for lattice paths restricted to rectangles and taking only north and east steps under the area statistic, given by the parameter $n$. They are also the number of $\ell$-dimensional subspaces of $\mathbb{F}_q^{\ell+m}$ and appear in many other generating functions as the $q$-analogue generalization of the ubiquitous  binomial coefficients. Notably, the numbers $N_n(\ell,m)$ form a symmetric unimodal sequence 
\[ 1 = N_0(\ell,m)\leq N_1(\ell,m) \leq \cdots \leq N_{\lfloor m\ell/2\rfloor}(\ell,m) \geq \cdots \geq N_{m\ell}(\ell,m)=1, \]
a fact conjectured by Cayley in 1856 and proven by Sylvester in 1878 via the representation theory of $sl_2$~\cite{Sylvester1878}.  One hundred forty years later, no previous asymptotic methods have been able to prove this unimodality.  

\subsection*{Asymptotics of $N_n(\ell,m)$}
Our first result is an asymptotic formula for $N_n (\ell,m)$ in 
the regime $\ell/m \to A$ and $n/m^2 \to B$ for any fixed $A > B > 0$. 
This is the regime in which a limit shape of the partitions exists: 
$\ell/m \to A$ implies the aspect ratio has a limit and 
$n/m^2 \to B \in (0,A)$ implies the portion of the $m \times \ell$ 
rectangle that is filled approaches a value that is neither zero nor one.  
By ``asymptotic formula'' we mean a formula giving $N_n(\ell,m)$ 
up to a factor of $1 + o(1)$; such asymptotic equivalence is denoted 
with the symbol $\sim$.  By the obvious symmetry $N_{n}(\ell,m) = 
N_{m\ell -n}(\ell,m)$ it suffices to consider only the case $A \geq 2B > 0$.

To state our results, given $A \geq 2B > 0$ we define three quantities 
$c, d$ and $\Delta$.  The quantities $c$ and $d$ are the unique positive 
real solutions (see Lemma~\ref{lem:uniquecd}) to the simultaneous equations 
\begin{align} 
A &= \int_0^1 \frac{1}{1-e^{-c-dt}}dt - 1 = \frac{1}{d} 
   \log \left( \frac{e^{c+d}-1}{e^c-1}\right) - 1  \, , 
   \label{eq:integrals A} \\
B &= \int_0^1 \frac{t}{1-e^{-c-dt}}dt - \frac{1}{2} 
   = \frac{d\log(1-e^{-c-d})+\dilog(1-e^{-c})-\dilog(1-e^{-c-d})}{d^2}
   \label{eq:integrals B} \, ,
\end{align}
where we recall the dilogarithm function 
\[ \dilog(x) = \int_1^x \frac{\log t}{1-t}dt 
   = \sum_{k=1}^{\infty} \frac{(1-x)^k}{k^2} \] 
for $|x-1|<1$.  The quantity $\Delta$, which will be seen to be 
strictly positive, is defined by 
\begin{equation} 
\label{eq:Delta}
\Delta = 
\frac{2Be^c(e^d-1) + 2A(e^c-1)-1}{d^2(e^{d+c}-1)(e^c-1)} - \frac{A^2}{d^2}\, .
\end{equation}

\begin{thm} \label{th:main}
Given $m, \ell$ and $n$, let $A := \ell/m$ and $B := n/m^2$ 
and define $c, d$ and $\Delta$ as above.  Let $K$ be any 
compact subset of $\{ (x,y) : x \geq 2y > 0 \}$.  
As $m \to \infty$ with $\ell$ and $n$ varying so that 
$(A,B)$ remains in $K$, 
\begin{equation} \label{eq:N}
N_n(\ell,m) \sim 
   \frac{e^{m \left[cA + 2dB - \log(1-e^{-c-d}) \right ]}}
   {2\pi m^2 \sqrt{\Delta \, (1-e^{-c})\, (1-e^{-c-d})}} \, ,
\end{equation}
where $c$ and $d$ vary in a Lipschitz manner with $(A,B) \in K$.
\end{thm}

\begin{unremark}
In the special case $B = A/2$, the parameters take on the elementary values
$$ d = 0 \, , \qquad c = \log \left(\frac{A+1}{A}\right) \, , \qquad 
   \mbox{ and } \qquad \Delta = \frac{A^2 (A+1)^2}{12} \, . $$
In this case we understand the exponent and leading constant to be 
their limits as $d \to 0$, giving
$$ N_{Am^2/2} (Am,m) \sim \frac{\sqrt{3}}{A\pi m^2} 
   \left [ \frac{(A+1)^{A+1}}{A^A} \right ]^m \, . $$

The special case when $A \to \infty$, so that $N_n(\ell,m) = N_n(m)$ 
and the restriction on partition sizes is removed corresponds to 
$c=0$ and $d$ is a solution to an  explicit equation given in 
Lemma~\ref{lem:uniquecd}.  In this case the result matches the 
one obtained first by Szekeres~\cite{Szekeres1953} using complex analysis,
then by Canfield~\cite{Canfield1997} using a recursion, and most recently
by Romik~\cite{Romik} using probabilistic methods based on 
Fristedt's ensemble~\cite{Fristedt1993}.  These works and others
are further explained in Section~\ref{sec:history}.
\end{unremark}

\subsection*{Unimodality}

Our second result gives an asymptotic estimate of the consecutive differences 
of $N_n$.  In fact our motivation for deriving more accurate asymptotics
for $N_n (\ell, m)$ was to be able to analyze the sequence 
$\{ N_{n+1} (\ell,m) - N_n(\ell , m) : n \geq 1 \}$.
Sylvester's proof of unimodality of $N_n (\ell,m)$ in 
$n$~\cite{Sylvester1878}, and most subsequent 
proofs~\cite{Stanley1983,Stanley1985,Proctor1982}, are algebraic, 
viewing $N_n(\ell,m)$ as dimensions of certain vector spaces, or 
their differences as multiplicities of representations. 
While there are also purely combinatorial proofs of unimodality, 
notably O'Hara's~\cite{OHara1990} and the more abstract one 
in~\cite{PouzetRosenberg}, they do not give the desired symmetric 
chain decomposition of the subposet of the partition lattice. 
These methods do not give ways of estimating the asymptotic size 
of the coefficients or their difference.  It is now known that 
$N_n (\ell , m)$ is strictly unimodal~\cite{PakPanova2013}, 
and the following lower bound on the consecutive difference was 
obtained in~\cite[Theorem 1.2]{PakPanova2017} using a connection 
between integer partitions and Kronecker coefficients: 
\begin{equation} \label{eq:pak-panova}
N_n(\ell,m)-N_{n-1}(\ell,m) \geq 0.004 \frac{2^{\sqrt{s}}}{s^{9/4}},
\end{equation}
where $n\leq\ell m/2$ and $s=\min\{2n,\ell^2,m^2\}$.  In particular, 
when $\ell=m$ we have $s=2n$.  

Any sharp asymptotics of the difference appears to be out of reach of 
the algebraic methods in this previous series of papers.  Refining
Theorem~\ref{th:main} we are able to obtain the following estimate.

\begin{thm} \label{th:difference}
Given $m, \ell$ and $n$, let $A := \ell/m$ and $B := n/m^2$ and define 
$d$ as above. Suppose $m, \ell$ and $n$ go to infinity so that $(A,B)$ 
remains in a compact subset $K$ of $\{ (x,y) : x \geq 2y > 0 \}$ and 
$$ m^{-1} \, |n-lm/2| \rightarrow \infty . $$
Then
$$ N_{n+1}(\ell,m) - N_n(\ell,m) \sim \frac{d}{m}N_n(\ell,m). $$
\end{thm}

\begin{unremark}
The condition $m^{-1} \, |n-lm/2| \rightarrow \infty$ is equivalent to 
$m \, \left|A-B/2\right| \rightarrow \infty$ and is satisfied if and 
only if $d$, which depends on $m$, is not $O(m^{-1})$.  It is 
automatically satisfied whenever the compact set $K$ is a subset 
of $\{ (x,y) : x > 2y > 0 \}$. 
\end{unremark}

\subsection*{Corollary: Asymptotics of Kronecker coefficients}
 
Recent developments in the representation theory of the symmetric 
and general linear groups, motivated by applications in 
Computational Complexity theory, have realized the consecutive 
differences  $N_{n+1}(\ell,m)-N_n(\ell,m)$ as specific Kronecker 
coefficients for the tensor product of irreducible $S_{\ell m}$ 
representations (see, for instance,~\cite{PakPanova2013} which is also 
one of the unimodality proofs).  The Kronecker coefficient 
$g(\la,\mu,\nu) = \dim {\rm Hom} (\mathbb{S}_\la, \mathbb{S}_\mu 
\otimes \mathbb{S}_\nu)$ is the multiplicity of the irreducible 
$S_{|\la|}$ Specht module $\mathbb{S}_\la$ in the tensor product 
of two other irreducible representations.  It is a notoriously 
hard problem to determine the values of these coefficients, 
and their combinatorial interpretation has been an outstanding 
open problem in Algebraic Combinatorics (see Stanley~\cite{StanleyOP}) 
since their definition by Murnaghan in 1938.  In general, 
determining even whether they are nonzero is an NP-hard problem 
and it is not known whether computing them lies in NP. 
See~\cite{IP17} and the literature 
therein for some recent developments on the relevance of Kronecker coefficients
in distinguishing complexity classes on the way towards 
${\rm P} \neq {\rm NP}$.  Being able to estimate particular 
values of Kronecker coefficients is crucial to the 
Geometric Complexity Theory approach towards these problems.

Because it is known (see~\cite{PakPanova2013}) that the consecutive difference
$N_n(m,\ell) -N_{n-1}(m,\ell)$ equals the Kronecker coefficient
$g((m\ell -n,n), (m^\ell), (m^\ell)),$
Theorem~\ref{th:difference} gives the first tight asymptotic 
estimate on this family of Kronecker coefficients. 

\begin{cor}
The Kronecker coefficient of $S_{m\ell}$ for the (rectangle, 
rectangle, two-row) case is asymptotically given by
$$g( (m^\ell), (m^\ell), (m\ell-n-1,n+1)) 
   = N_{n+1}(\ell,m) - N_{n}(\ell,m) 
   \sim \frac{d}{m} N_n(\ell,m)
   \sim \frac{de^{m \left[cA + 2dB - \log(1-e^{-c-d}) \right ]}}
   {2\pi m^3 \sqrt{\Delta \, (1-e^{-c})\, (1-e^{-c-d})}},$$
with constants and ranges as in Theorems~\ref{th:main} and~\ref{th:difference}. 
\end{cor}

\section{Review of previous results and description of methods}
\label{sec:history}

\subsection{Combinatorial Enumeration}

Work on this problem has developed in two streams.  First, there have
been combinatorial results aimed at asymptotic enumeration in various
regimes.  After Hardy and Ramanujan obtained an asymptotic formula 
for $N_n$ in~\cite{HardyRamanujan1918}, enumerative work focused on
$N_n (m)$, the number of partitions with part sizes bounded by $m$, or
equivalently, partitions of $n$ that fit in an $m \times \infty$ strip.
In~1941, Erd{\"o}s and Lehner~\cite{ErdosLehner1941} showed that 
$N_n(m) \sim \frac{n^{m-1}}{m!(m-1)!}$ for $m = o(n^{1/3})$.
This was generalized by Szekeres and others, ultimately leading 
to asymptotics of $N_n(m)$ for all $m$ in 1953~\cite{Szekeres1953}.  
Szekeres simplified his arguments a number of times, ultimately 
giving asymptotics using only a saddle-point analysis, without 
needing results on modular functions; his argument has been referred 
to as the Szekeres circle method.  Canfield~\cite{Canfield1997} gave 
a completely elementary proof (no complex analysis) of asymptotics 
for $N_n(m)$ using a recursive formula satisfied by these numbers. 

The combinatorial stream contains a few results on the asymptotics
of $N_n (m , \ell)$ but only in the regime where $m$ and $\ell$ are
greater than $\sqrt{n}$ by at least a factor of $\log n$.  This is
a natural regime to study because the typical values of the
maximum part (equivalently the number of parts) of a partition of size $n$
was shown by Erd{\"o}s and Lehner~\cite{ErdosLehner1941} to be 
of order $\sqrt{n \log n}$.  Szekeres~\cite[Theorem 1]{Szekeres1990} 
used saddle-point techniques to express $N_n(\ell,m)$ in terms of 
$N_n$, $\lambda := \frac{\pi \ell}{\sqrt{6n}}$ and 
$\mu := \frac{\pi m}{\sqrt{6n}}$. If, in fact, 
$$\frac{\sqrt{6n}}{\pi}\left(\frac{1}{4} + \ee \right) \log n < 
   \ell,m < \frac{\sqrt{6n} \, \log n}{\pi} $$
for some $\ee > 0$, then the distributions defined by 
$\ell$ and $m$ are independent and equal, and Szekeres' formula simplifies to
$$ N_n(\ell,m) \sim N_n \, \exp\left[-(\lambda + \mu) - \sqrt{\frac{6n}{\pi}}\left(e^{-\lambda}+e^{-\mu}\right)\right]. $$
The Szekeres circle method was recently revisited by Richmond~\cite{Richmond2018}.  
In~\cite{JiangWang} the authors, independently and concurrently with 
our paper, used the generating function for $q$-binomial coefficients 
and a saddle point analysis to derive the asymptotics for 
$N_n (m , \ell)$ in the cases when $m,\ell \geq 4\sqrt{n}$, 
corresponding to $B \leq \min\{1,A^2\}/16$ in our notation. 
Those authors express their result using the root of a 
hypergeometric identity similar to~\eqref{eq:integrals B}, 
however their methods give weaker error bounds and consequently 
cannot answer questions of unimodality.

\subsection{Probabilistic limit theorems}

The second strand of work on this problem has been probabilistic.
The goal in this strand has been to determine properties of a
random partition or Young diagram, picked from a suitable 
probability measure.  This approach goes back at least to Mann and 
Whitney~\cite{MannWhitney1947}, who showed that the size of 
a uniform random partition contained in an $\ell \times m$ 
rectangle satisfies a normal distribution.  
Frisdtedt~\cite{Fristedt1993} defined a distribution on partitions 
of all sizes, weighted with respect to a parameter $q < 1$.
The key property of the measure employed is that it makes the 
number $X_k(\lambda)$ of parts of size $k$ in the partition $\lambda$ 
drawn under this distribution independent as $k$ varies;
the distributions of the $X_k$ are reduced geometrics with respective parameter 
$1 - q^k$, so that their mean is $q^k / (1 - q^k)$.  
Fristedt is chiefly concerned with the limiting behavior of $k X_k$ for $k = o(\sqrt{n})$, which 
rescales, on division by $\sqrt{n}$, to an exponential distribution.

Much of the work following Fristedt's is concerned with a description of 
the limiting shape of the random partition, and fluctuations around that 
shape.  The limit shape of an unrestricted partition was posed as a 
problem by Vershik and first answered in~\cite{SzalayTuran1,SzalayTuran2}.
In 2001, Vershik and Yakubovich~\cite{VershikYakubovich2001} 
describe the limit shape for singly restricted partitions: those
with $m \leq c \sqrt{n}$.  They obtain both main (strong law)
results and fluctuation (CLT) results.  It is in this paper that
the probability measures $\P_m$ used in our analysis below first
arose, although we were unaware of this when we first derived them
from large deviation principles.  The limit shape for doubly restricted
partitions in the regime $m, \ell = \Theta (\sqrt{n})$ was first
described by Petrov~\cite{Petrov2009}.  It is identified there with 
a portion of the curve $e^{-x} + e^{-y} = 1$, which represents the 
limit shape of unrestricted partitions.  More recently, 
Beltoft et al.~\cite{BeltoftBoutillierEnriquez2012} obtained
fluctuation results in the doubly restricted regime.  The 
limiting fluctuation process is an Ornstein-Uhlenbeck bridge, 
generalizing the two-sided stationary Ornstein-Uhlenbeck process 
that gives the limiting fluctuations in the unrestricted 
case~\cite{VershikYakubovich2001}.   

\subsection{Enumeration via probability}

Strangely, we know of only one paper combining these two streams.
Tak{\'a}cs~\cite{Takacs1986} observed the following consequence of
the work of Fristedt and others.  
Begin a discrete walk at $(\ell,0)$ and randomly choose steps in the $(0,-1)$
or $(-1,0)$ directions by making independent fair coin flips.  If this walk 
goes from $(\ell,0)$ to $(0,-m)$ it takes precisely $m + \ell$ steps and 
encloses a Young diagram fitting in an $m \times \ell$ rectangle: see Figure~\ref{fig:RW}.
Let $G(m,\ell)$ denote the event that a walk of length $m+\ell$ ends at $(0,-m)$ 
and let $H(m,n)$ denote the event that 
the resulting Young diagram has area $n$.  Under the IID fair coin
flip probability measure on paths, all paths of length $m + \ell$ have 
the same probability $2^{-(m+\ell)}$.  Therefore, $\P [G(m,\ell) \cap
H(m,n)] = 2^{-(m+\ell)} N_n(\ell , m)$ and the problem of counting
$N_n(\ell,m)$ is reduced to determining the probability $\P [G(m,\ell) \cap H(m,n)]$.

\begin{figure}[!ht]
\centering
\includegraphics[height=1.6in]{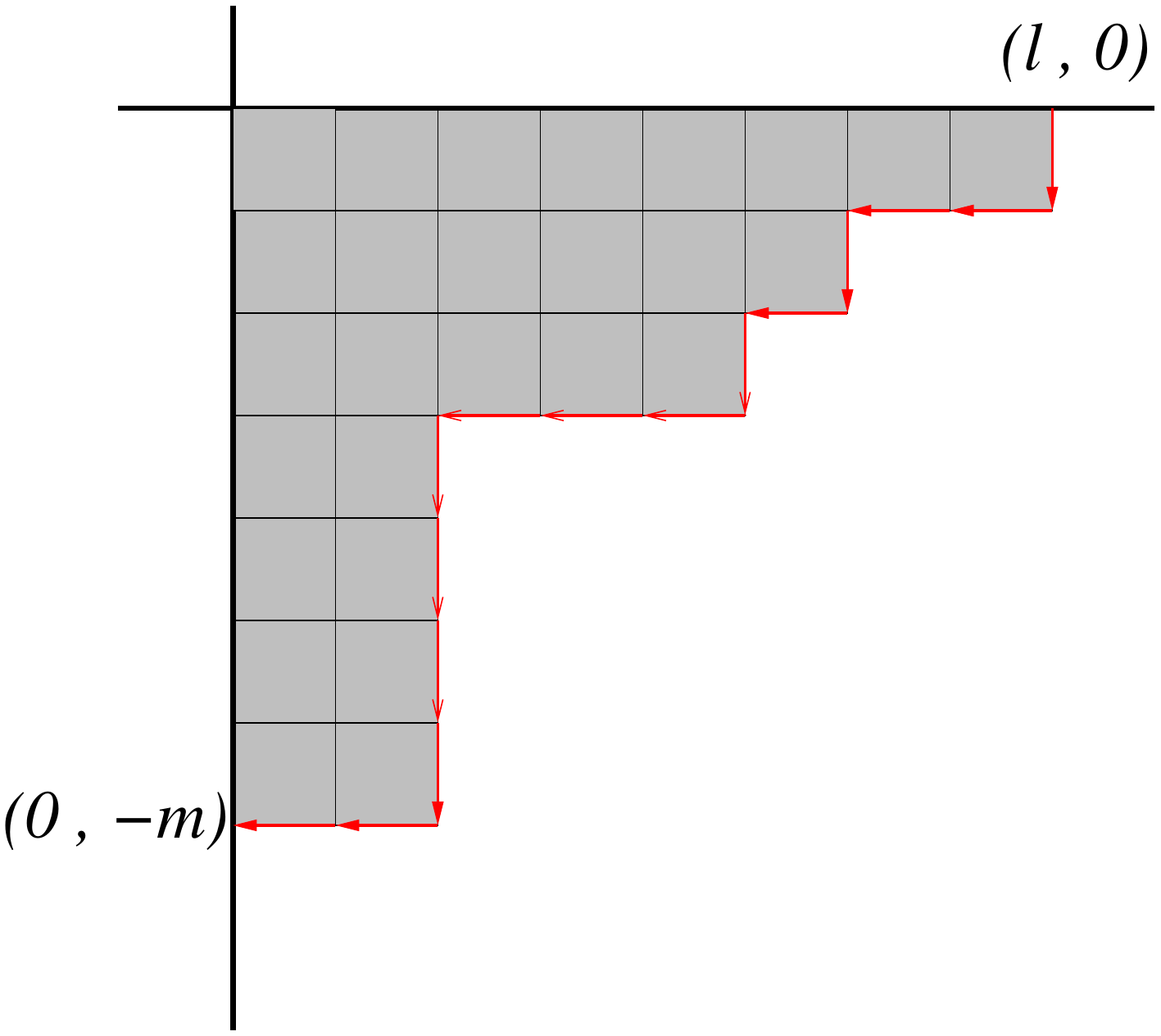}
\caption{The red arrows are the steps in a South and West directed simple random walk}
\label{fig:RW}
\end{figure}

Tak{\'a}cs observed that this probability is computable by a two-dimensional
local central limit theorem, ultimately obtaining 
bounds on the relative error that are of order $(m + \ell)^{-3}$. 
These error bounds are meaningful when $n$ differs from $m \ell / 2$ 
by up to a few multiples of $\log (m+\ell)$ standard deviations:
if $\ell = \theta (m)$ this means that $|B - A/2|m^2 = \Theta 
(m^{3/2} \log m)$.  
When $|B - A/2| \gg m^{-1/2} \log m$ the error is much bigger 
than the main term of the Gaussian estimate provided by the LCLT
and one cannot recover meaningful information about $N_n(\ell , m)$.
This is where Tak{\'a}cs left off and the present manuscript picks up.

\subsection{Description of our methods}

We use a local large deviation computation in place of a local
central limit theorem: this is possible because the restriction to an $m \times \ell$
rectangle is a linear constraint.  Indeed, consider now 
a partition $\lambda=(\lambda_1,\dots,\lambda_m)$ with at most $m$ parts
(so some $\lambda_j$ may be zero) and define $\lambda_0 := \ell$ and $\lambda_{m+1} := 0$.
It is convenient to encode a partition with respect to its \emph{gaps} $x_j := \lambda_j - \lambda_{j+1}$,
so the condition that $\lambda$ be a partition of $n$ of size at most $\ell$ is equivalent to $x_j \geq 0$ and
\begin{equation} \label{eq:diagram}
\sum_{j=0}^m x_j = \ell \,, \qquad\qquad \sum_{j=0}^m j x_j = n \, .
\end{equation}
Figure~\ref{fig:1} gives a pictorial proof.

\begin{figure}[!ht]
\centering
\includegraphics[width=4in]{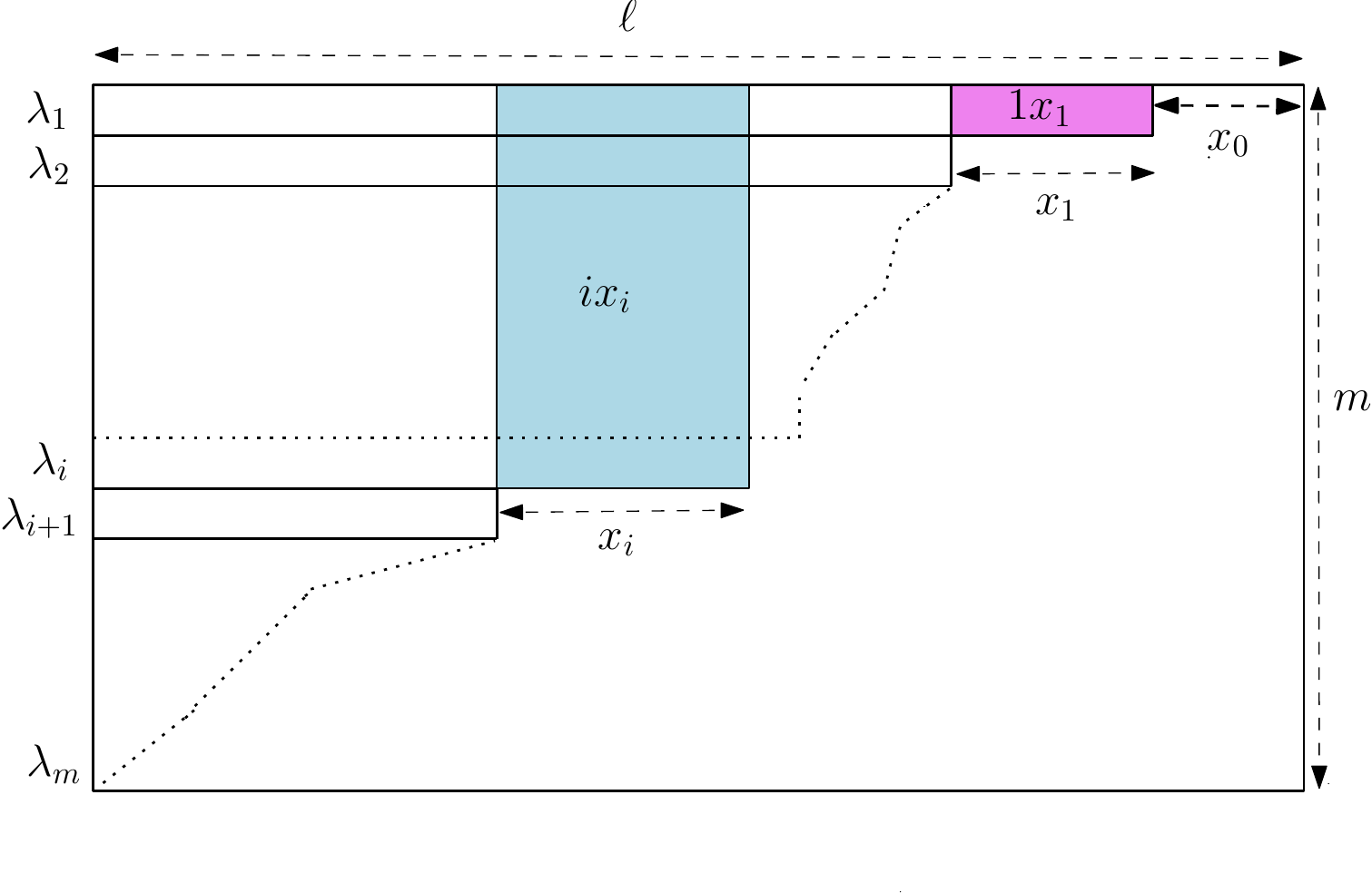}
\caption{The total area $n$ of a partition is composed of rectangles of area $j x_j$}
\label{fig:1}
\end{figure}

Solving the large deviation problem produces a ``tilted measure''
in which the gaps $X_j$ are no longer IID reduced geometrics 
with parameter $1/2$ but are instead given by independent 
reduced geometric variables whose parameters $q_j = 1 - p_j$
vary in a log-linear manner.  Log-linearity is dictated by the
variational large deviation problem and leads to the same 
simplification as before.  Not all partitions have the same probability
under the tilted measure, but all those resulting in a given value
of $\ell$ and $n$ do have the same probability.  Lastly, one must
choose the particular linear function $\log q_j = -c - d (j/m)$ to
ensure that $\lambda$ being a partition of $n$ with parts of size at most $\ell$ 
will again be in 
the central part of the tilted measure, so that asymptotics 
can be read off from a local CLT for the tilted measure.

The tilted measures $\P_m$ that we employ are denoted $\mu_{x,y}$
in~\cite{VershikYakubovich2001} and referred to there
as the grand ensemble of partitions.  That paper, however, was
not concerned with enumeration, only with limit shape results. 
For this reason they do not state or prove enumeration results.
In fact~\cite{Petrov2009} is able to prove the shape result
by estimating exponential rates only, showing rather elegantly
that an $\ee$ error in the rescaled shape leads to an exponential 
decrease in the number of partitions.
The present manuscript combines the idea of the grand ensemble
with some precise central limit estimates and some algebra inverting
the relation between the log-linear parameters and the parameters
$A$ and $B$ defining the respective limits of $\ell/m$ and $n/m^2$
to give estimates on $N_n(\ell , m)$ precise enough also to yield
asymptotic estimates on $N_{n+1} (\ell , m) - N_n (\ell , m)$.

The first step of carrying this out necessarily recovers the 
leading exponential behavior for $N_n (\ell , m)$, which is
implicit  in~\cite{VershikYakubovich2001} and~\cite{Petrov2009}
though Petrov only states it as an upper bound.
Interestingly, Tak{\'a}cs did not seem to be aware of the
ease with which the exponential rate may be obtained.  His result
states a Gaussian estimate and an error term.  As noted above, it is nontrivial
only when the $(m+\ell)^{-3}$ relative error term does not swamp 
the main terms, which occurs when $n$ is close to $\ell m/2$ 
(see also~\cite{AlmkvistAndrews1991}).  Figure~\ref{fig:expgrow}
shows Tak{\'a}cs' predicted exponential growth rate on a family of examples compared to the
actual exponential growth rate that follows from Theorem~\ref{th:main}.

\begin{figure}[!ht]
\centering
\includegraphics[width=0.4\linewidth]{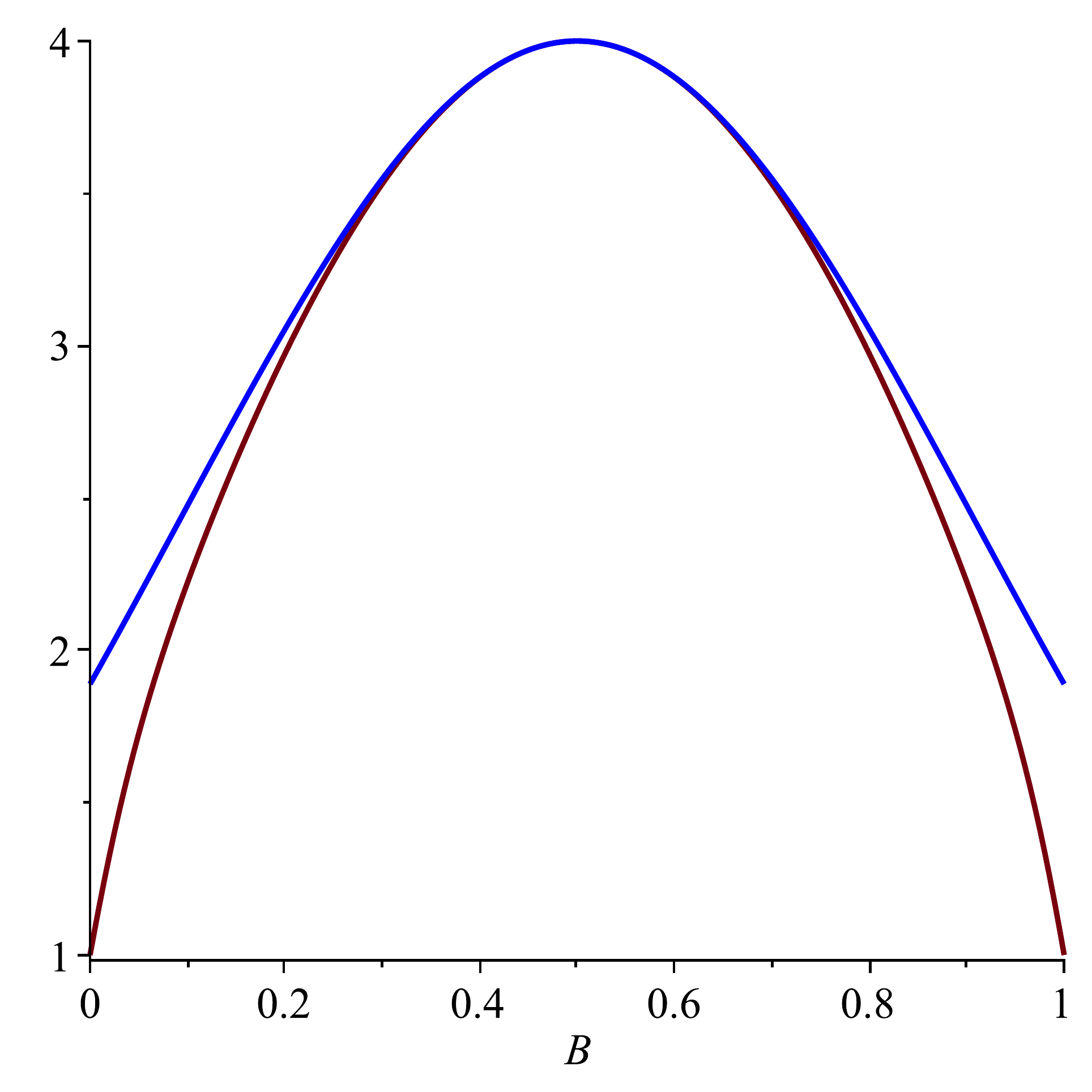}
\caption{Exponential growth of $N_{Bm^2}(m,m)$ predicted by Tak{\'a}cs' formula (blue, above) compared to the actual exponential growth given by Theorem~\ref{th:main} (red, below).}
\label{fig:expgrow}
\end{figure}

\section{A discretized analogue to Theorem~\protect{\ref{th:main}}}
\label{sec:discrete}

We now implement this program to derive asymptotics. With $c_m$ and $d_m$ to be specified later, let $q_j := e^{-c_m - j d_m/m}$, let $p_j := 1-q_j$ and let 
\[ L_m := \sum_{j=0}^m \log p_j.\]  
Let $\P_m$ be a probability law making the random variables $\{ X_j : 0 \leq j \leq m \}$ independent reduced geometrics with respective parameters $p_j$.  Define random variables $S_m$ and $T_m$ by
\begin{equation} 
\label{eq:S and T}
S_m := \sum_{i=0}^m X_i \, ; \qquad T_m:=\sum_{i=1}^m i X_i \, ,
\end{equation}
corresponding to the unique partition $\lambda$ satisfying $X_j = \lambda_j - \lambda_{j+1}$.
We first prove a result similar to Theorem~\ref{th:main}, except that the parameters $c$ and $d$ that solve integral Equations~\eqref{eq:integrals A} and~\eqref{eq:integrals B} are replaced by $c_m$ and $d_m$ satisfying the discrete summation Equations~\eqref{eq:mu_m} and~\eqref{eq:nu_m} below.  These equations say that $\E S_m = \ell$ and $\E T_m = m$.  Writing this out, using $\disp \E X_j = 1/p_j - 1 = 1 / \left ( 1 - e^{-c_m - d_m j/m} \right ) - 1$, gives
\begin{align}
\ell & =  \sum_{j=0}^m \frac{1}{1 - e^{-c_m - d_m j/m}} - (m+1) \label{eq:mu_m} \\
n & = m \sum_{j=0}^m \frac{j/m}{1 - e^{-c_m - d_m j/m}} - \frac{m(m+1)}{2} \label{eq:nu_m}  \, .
\end{align}

Let $M_m$ denote the covariance matrix for $(S_m , T_m)$.  The entries may be computed from the basic identity $\Var (X_j) = q_j / p_j^2$, resulting in
\begin{align}
\Var (S_m) & = \sum_{j=0}^m \frac{e^{-c_m - d_m j/m}}
   {\left ( 1 - e^{-c_m - d_m j/m} \right )^2} \label{eq:alpha_m} \\
\Cov (S_m , T_m) & = \sum_{j=0}^m j \frac{e^{-c_m - d_m j/m}}
   {\left ( 1 - e^{-c_m - d_m j/m} \right )^2} \label{eq:beta_m} \\
\Var (T_m) & = \sum_{j=0}^m j^2 \frac{e^{-c_m - d_m j/m}}
   {\left ( 1 - e^{-c_m - d_m j/m} \right )^2} \label{eq:gamma_m} \, .
\end{align}

\begin{thm}[discretized analogue] \label{th:1'}
Let $c_m$ and $d_m$ satisfy~\eqref{eq:mu_m}~--~\eqref{eq:nu_m}. Define $\alpha_m, \beta_m$ and $\gamma_n$ to be the normalized
entries of the covariance matrix
\[
	\alpha_m := m^{-1} \Var(S_m) \; ; \qquad
   	\beta_m := m^{-2} \Cov(S_m , T_m) \; ; \qquad
   	\gamma_m := m^{-3} \Var(T_m) \; ,
\]
which are $O(1)$ as $m \rightarrow\infty$. Again, let $A := \ell / m$ and $B := n/m^2$ and $\Delta_m := \alpha_m
\gamma_m - \beta_m^2$. Then as $m \to \infty$ with $\ell$ and $n$ varying so that $(A,B)$ remains in a compact subset of $\{ (x,y) : x \geq 2y > 0 \}$,
\begin{equation} 
\label{eq:N'}
N_n(\ell,m) \sim \frac{1}{2\pi m^2 \sqrt{\Delta_m}} \exp \left \{ m \left( -\frac{L_m}{m} + c_m A + d_m B \right ) \right \} \, .
\end{equation}
\end{thm}

\begin{proof}
The atomic probabilities $\P_m (\XX = \xx)$ depend only on $S_m$ and $T_m$ as
\begin{align*}
\log \P_m (\XX = \xx) & = \sum_{j=0}^m \left(\log p_j + x_j \log q_j\right) \\
& = L_m - \sum_{j=0}^m \left ( c_m + j \frac{d_m}{m} \right ) x_j \\
& = L_m - c_m \left ( \sum_{j=0}^m x_j \right ) - \frac{d_m}{m} 
   \left ( \sum_{j=0}^m j x_j \right ).
\end{align*}
In particular, for any $\xx$ satisfying~\eqref{eq:diagram},
\begin{equation} \label{eq:xx}
\log \P (\XX = \xx) = L_m - c_m \ell - \frac{d_m}{m} n \, .
\end{equation}
Three things are equivalent: $(i)$ the vector $\XX$ satisfies the 
identities~\eqref{eq:diagram}; $(ii)$ the pair $(S_m , T_m)$ is equal to 
$(\ell,n)$; $(iii)$ the partition $\lambda= (\lambda_1 , \ldots , \lambda_m)$ 
defined by $\lambda_j - \lambda_{j+1} = X_j$ for $2 \leq j \leq m-1$, 
together with $\lambda_1 = \ell - X_0$ and $\lambda_m = X_m$, is a 
partition of $n$ fitting inside a $m \times \ell$ rectangle.  Thus,
\begin{align}
N_n(\ell , m) &= \P_m \left [ (S_m,T_m) = (\ell , n) \right ] \, 
   \exp \left(- L_m + c_m \ell + \frac{d_m}{m} n\right) \nonumber \\
&= \P_m \left [ (S_m,T_m) = (\ell , n) \right ] \, \exp \left [ m \left ( 
   -\frac{L_m}{m} + c_m A + d_m B \right ) \right ] \label{eq:P_m} \, .
\end{align}
Comparing~\eqref{eq:N'} to~\eqref{eq:P_m}, the proof is completed by 
an application of the LCLT in Lemma~\ref{lem:lclt}.
\end{proof}

Lemma~\ref{lem:lclt} is stated for an arbitrary sequence of parameters 
$p_0, \ldots , p_m$ bounded away from 0 and 1, though we need it only 
for $p_j = 1 - e^{-c_m - d_m j/m}$.  For a $2 \times 2$ matrix $M$, 
denote by $M(s,t):= [s\, , \,\, t] \, M\, [s\,,\,\,t]^T$ the 
corresponding quadratic form.

\begin{lem}[LCLT]\label{lem:lclt}
Fix $0 < \delta < 1$ and let $p_0, \ldots , p_m$ be any real numbers in the interval $[\delta , 1 - \delta]$.  Let $\{ X_j \}$ be independent reduced geometrics with respective parameters $\{ p_j \}$, $S_m := \sum_{j=0}^m X_j,$ and $T_m := \sum_{j=0}^m j X_j$. Let $M_m$ be the covariance matrix for $(S_m , T_m)$, written
\[ M_m = \disp \left ( \begin{array}{cc} \alpha_m m & \beta_m m^2 \\ \beta_m m^2 & \gamma_m m^3 \end{array} \right ) \, ,\]
$Q_m$ denote the inverse matrix to $M_m$, and $\Delta_m = m^{-4} \det M_m = \alpha_m \gamma_m - \beta_m^2$.   Let $\mu_m$ and $\nu_m$ denote the respective means $\E S_m$ and $\E T_m$. Denote $p_m(a,b) := \P ((S_m , T_m) = (a,b))$.  Then
\begin{equation} 
\label{eq:LCLT}
\sup_{a,b \in \Z} m^2 \left | p_m (a,b) - \frac{1}{2 \pi (\det M_m)^{1/2}} e^{-\frac{1}{2} Q_m (a - \mu_m , b - \nu_m)} \right | \to 0
\end{equation}
as $m \to \infty$, uniformly in the parameters $\{ p_j \}$ in the allowed range.  In particular, if the sequence $(a_m , b_m)$ satisfies $Q_m (a_m - \mu_m , b_m - \nu_m) \to 0$ then 
\[ \P(S_m = a_m , T_m = b_m) = \frac{1}{2\pi \sqrt{\Delta_m} m^2} \left(1+O\left(m^{-3/2}\right)\right) \,. \]
\end{lem}

The following consequence will be used to prove Theorem~\ref{th:difference}.

\begin{cor}[LCLT consecutive differences]\label{cor:lclt_error}
Let $\mN_m(a,b) := \frac{1}{2 \pi (\det M_m)^{1/2}} e^{-\frac{1}{2} Q_m (a - \mu_m , b - \nu_m)}$ be the normal approximation in Equation~\eqref{eq:LCLT}. Using the notation of Lemma~\ref{lem:lclt}, 
\[ \sup_{a,b\in \mathbb{Z}} \bigg|p_m(a,b+1)-p_m(a,b) - \big(\mN_m(a,b+1)-\mN_m(a,b)\big) \bigg|  = O(m^{-4}). \]
\end{cor}  

The technical but unsurprising proofs of Lemma~\ref{lem:lclt} and Corollary~\ref{cor:lclt_error} are given in the Appendix at the end of this article.  

\section{Limit shape} \label{sec:shape}

Suppose a Young diagram is chosen uniformly from among all partitions
of $n$ fitting in a $m \times \ell$ rectangle. To simplify calculations, we imagine
this Young diagram outlining a compact set in the fourth quadrant of the plane 
and rotate $90^\circ$ counterclockwise to obtain a shape in the first quadrant. 
Let $\Xi_{n,m,\ell}$ denote the random set obtained in this manner after 
rescaling by a factor of $1/m$, so that the length in the
positive $x$-direction is bounded by~1.  Fix $A > 2B > 0$ and metrize
compact sets of $\R^2$ by the Hausdorff metric.  As $m \to \infty$ 
with $\ell/m \to A$ and $n/m^2 \to B$, the random set $\Xi_{n,m,\ell}$ 
converges in distribution to a deterministic set $\Xi^{A,B}$. See Figure~\ref{fig:limtcurve}
for some examples.

Our methods immediately recover the distributional convergence result
$\Xi_{n,m,\ell} \to \Xi^{A,B}$.  
As previously mentioned, this limit shape was known to 
Petrov~\cite{Petrov2009} and others.  Petrov identifies it 
as a portion of the limit curve for unrestricted partitions, 
which itself was posed as a problem by Vershik and answered 
in~\cite{SzalayTuran1,SzalayTuran2} (see also~\cite{Vershik96}). 
Because this result is already known, along with precise 
fluctuation information which we do not derive, we give only
the short argument here for distributional convergence.  We do not 
determine the best possible fluctuation results following from this method.

The shape $\Xi_{n,m,\ell}$ is determined by its boundary, a polygonal
path obtained from a partition $\lambda$ by filling in unit vertical connecting lines in the 
step function $x \mapsto m^{-1} \lambda_{\lfloor mx \rfloor}$.
Recall that the probability measure $\P_m$ restricted to the event
$\{ (S_m, T_m) = (\ell , n) \}$ gives all partitions counted by
$N_n(m , \ell)$ equal probability and that $\P_m$ gives the event
$\{ (S_m, T_m) = (\ell , n) \}$ probability $\Theta (m^{-2})$.
Distributional convergence of $\Xi_{n,m,\ell}$ to $\Xi^{A,B}$
then follows from the following.

\begin{pr} \label{pr:shape}
Fix $A > 2B > 0$.  Define the maximum discrepancy by
$$\MM := \max_{0 \leq j \leq m} \left | \sum_{i=0}^j \left ( 
   X_i - \frac{q_i}{p_i} \right ) \right | \, .$$
Then for any $\ee > 0$,
$$\P_m \left [ \MM \geq \ee m \right ] = o(m^{-2})$$
as $m \to \infty$ with $\ell/m \to A$ and $n/m \to B$.  
\end{pr}

\begin{proof}
This is a routine application of exponential moment bounds.  By our definition of $p_i$,
in this regime there exists $\delta>0$ such that $p_i \in 
[\delta , 1 - \delta]$ for all $i$.  Therefore, there are 
$\eta,K>0$ such that for $\lambda < \eta$, the mean zero variables 
$X_i - q_i / p_i$ all satisfy 
$\E \exp (\lambda (X_i - q_i / p_i)) \leq \exp (K \lambda^2)$.
Independence of the family $\{ X_i \}$ then gives 
$$\E \exp \left [ \lambda \sum_{i=0}^j (X_i - p_i / q_i) \right ]
   \leq e^{K m \lambda^2}$$
for all $j \leq m$.  By Markov's inequality,
$$\P ( |X_i - p_i / q_i| \geq \ee m) \leq e^{K m \lambda^2 - \lambda m} 
   \, .$$
Fixing $\lambda = 1/(2K)$ shows that this probability is bounded
above by $\exp (- m/(4K))$.  Hence, $\P (\MM \geq \ee m) \leq m e^{-m/(4K)}
= o(m^{-2})$ as desired.
\end{proof}

To see that Proposition~\ref{pr:shape} implies the limit shape statement,
let $\lambda_i := \ell - (X_0 + \cdots + X_{i-1})$ so that 
\[ y^{(m)}(i) := \E_m \lambda_i = \ell - \sum_{j=0}^{i-1} q_j / p_j.\]  
Proposition~\ref{pr:shape} shows the boundary of $\Xi_m$ to be within 
$o(m)$ of the step function $y^{(m)}(\cdot)$ except with probability 
$o(m^{-2})$. Since $\P_m$ restricted 
to the event $\{ (S_m, T_m) = (\ell , n) \}$ gives all partitions counted 
by $N_n(m , \ell)$ equal probability and $\P_m$ gives the event
$\{ (S_m, T_m) = (\ell , n) \}$ probability $\Theta (m^{-2})$,
the conditional law $(\P_m \, | \, (S_m , T_m) = (\ell , n))$
gives the event $\{ \MM > \ee m \}$ probability $o(1)$ as $m \to \infty$
with $\ell/m \to A$ and $n/m \to B$. Thus, the boundary of
$\Xi_m$ converges in distribution to the limit 
\begin{equation} \label{eq:y(x)}
y(x) := \lim_{m \to \infty} m^{-1} y^{(m)} (\lfloor mx \rfloor ) \, .
\end{equation}
Figure~\ref{fig:limtcurve} shows examples of two families of the limit curve as well as
a plot of the limit curve against uniformly generated restricted
partitions for several values of $m$ in the range $[120,300]$.

\begin{figure}[!ht]
\centering
\begin{subfigure}{0.315\textwidth}
\centering
\includegraphics[width=\linewidth]{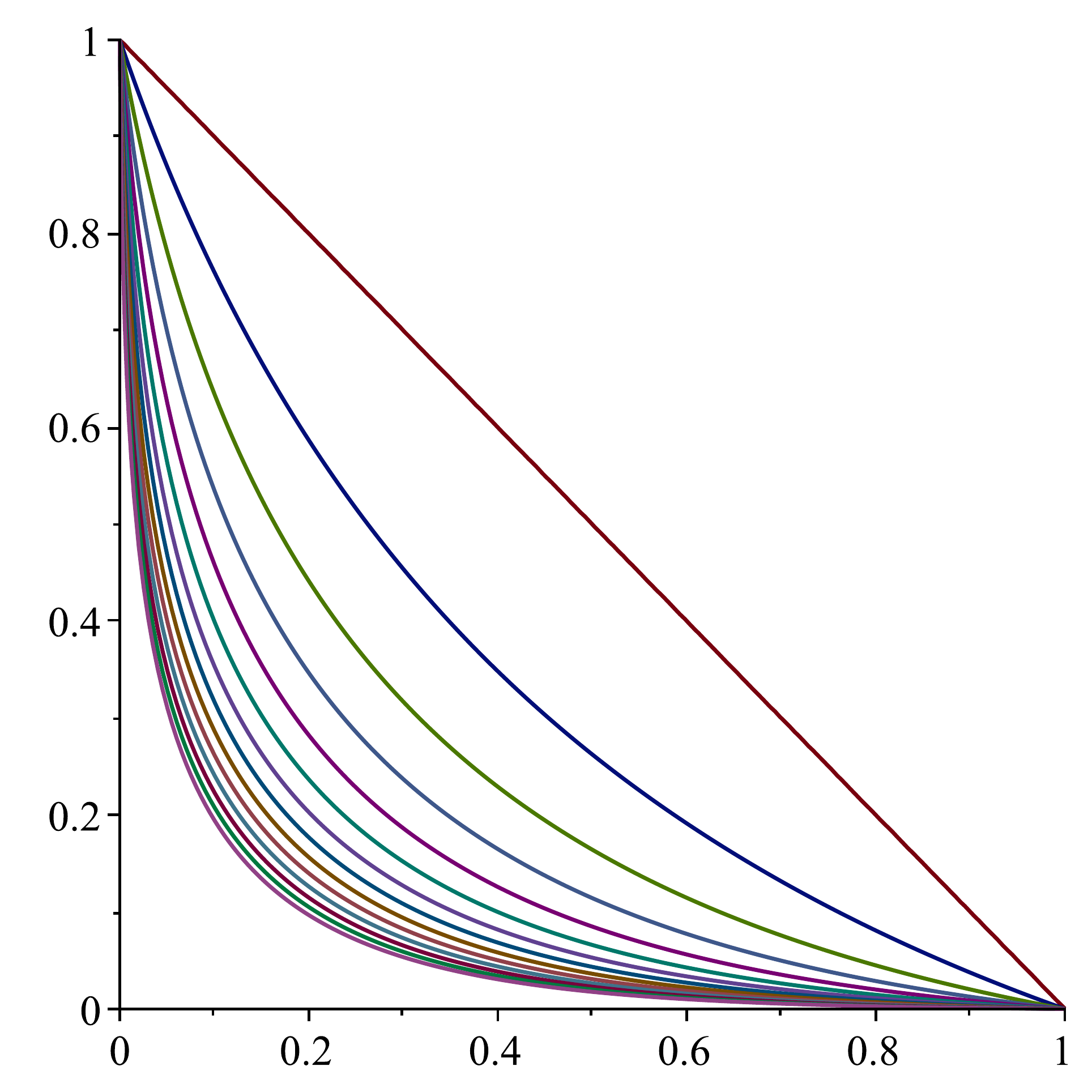} 
\subcaption*{$(A,B)=(1,1/k)$ \\ $k=2,\dots,15$}
\end{subfigure}
\begin{subfigure}{0.315\textwidth}
\centering
\includegraphics[width=\linewidth]{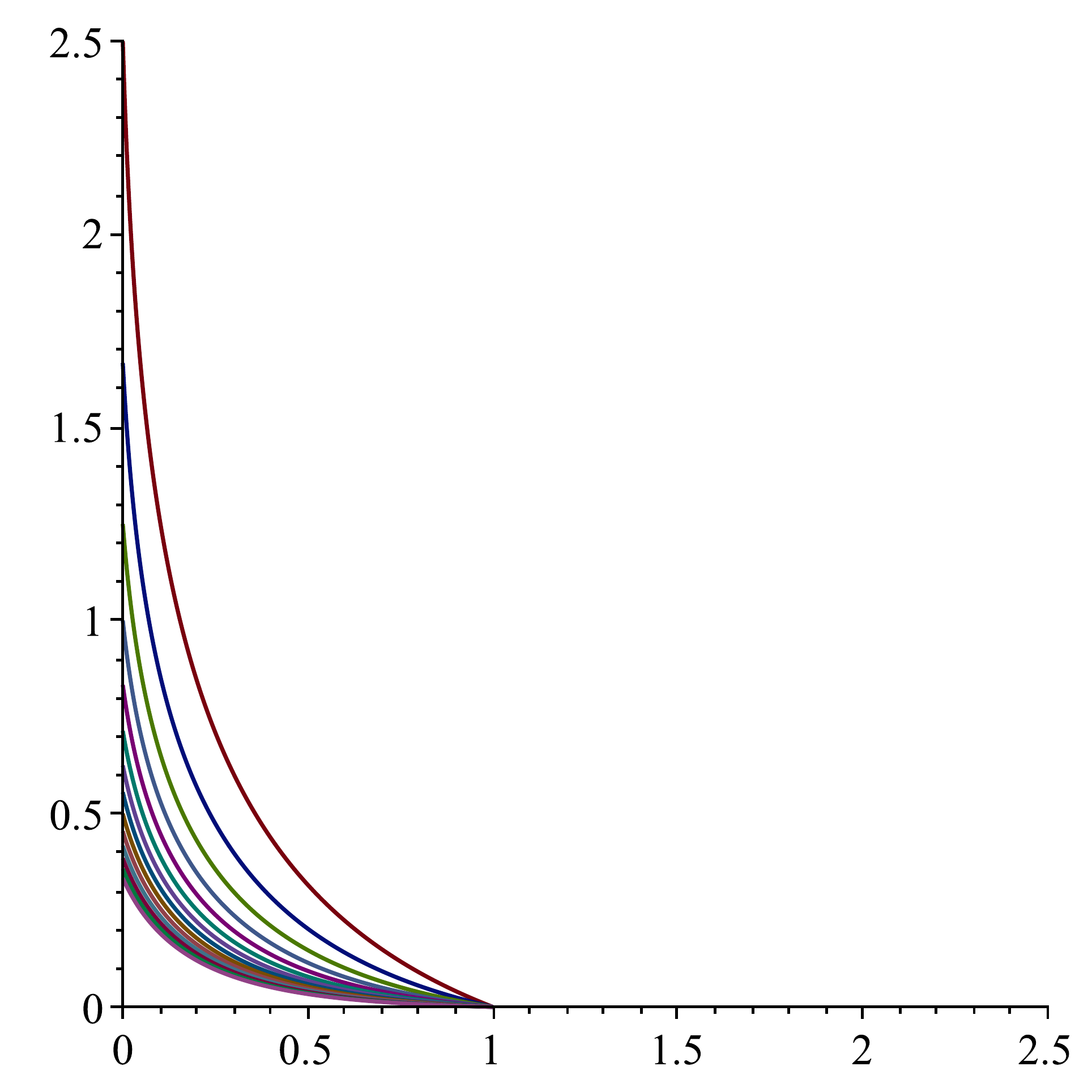}
\subcaption*{$(A,B)=(5/k,1/k)$ \\ $k=2,\dots,15$}
\end{subfigure}
\begin{subfigure}{0.35\textwidth}
\centering
\includegraphics[width=0.9\linewidth]{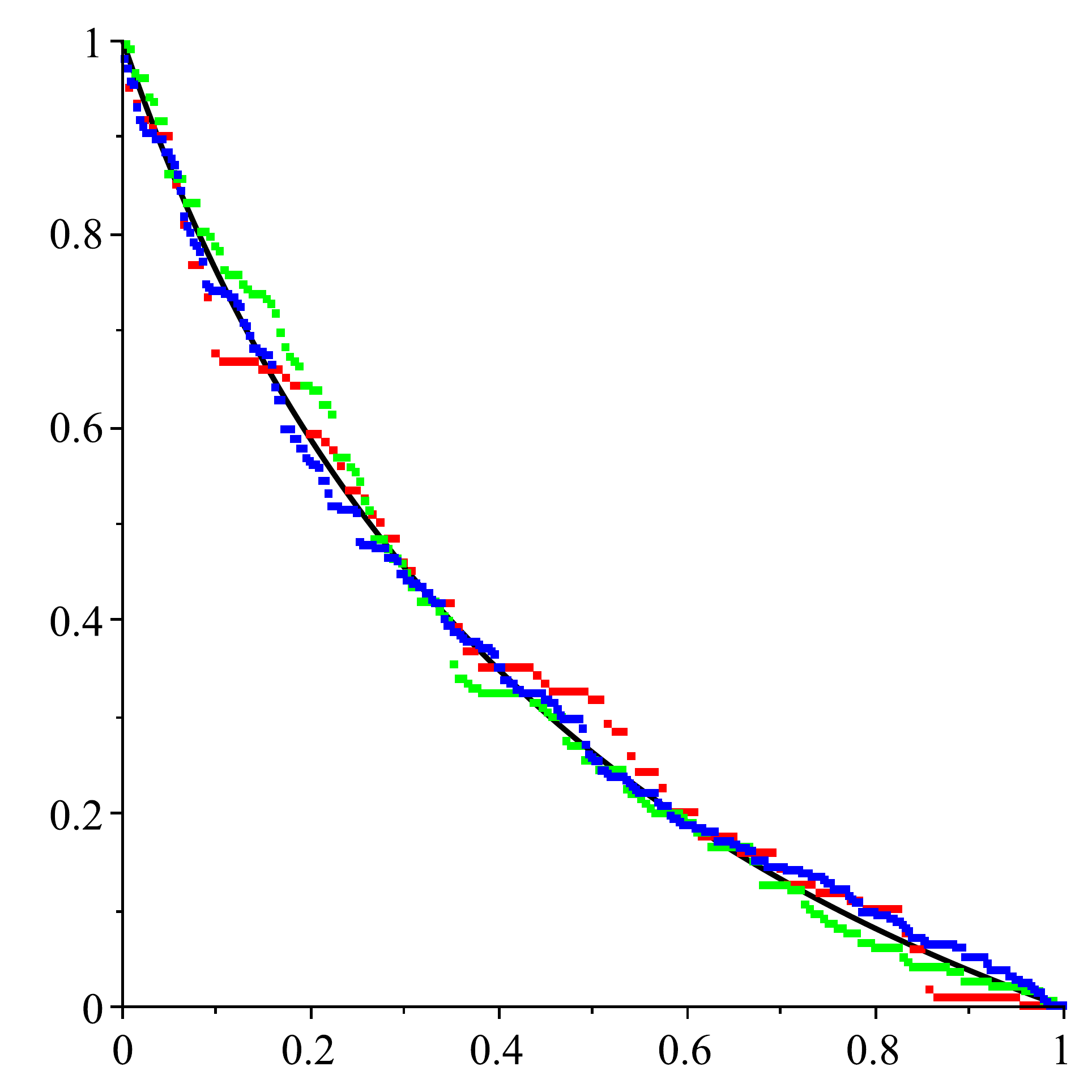}
\subcaption*{Limit curve of $(A,B)=(1,1/3)$ and random partitions of size {\color{red}120}, {\color{darkgreen}201} and {\color{blue}300}.}
\end{subfigure}
\caption{Limit shapes of scaled partitions as $m\rightarrow \infty$.}
\label{fig:limtcurve}
\end{figure}

Substituting the definition of $y^{(m)}(i)$ into~\eqref{eq:y(x)} and
evaluating the limit as an integral gives
$$y(x) = A + x - \int_0^x \frac{1}{1-e^{-c-dt}}dt 
   = A+x - \frac{1}{d}\ln\left(\frac{e^{xd+c}-1}{e^c-1}\right) \, .$$
After expressing $c$ in terms of $d$, this may be written implicitly as
$$e^{(A+1)d}-1 = (e^d-1) e^{d(A-y)} + (e^{Ad}-1)e^{d(1-x)}$$
which simplifies to
\begin{equation} \label{eq:lim_curve}
(1-e^{-c})e^{d(A-y)} + e^{-c}e^{-dx} = 1 
\end{equation}
as long as $A > 2B$; in the special case $A = 2B$ one obtains simply 
$y = A \cdot (1-x)$.

It is worth comparing this result with the limit shape derived 
in~\cite{Petrov2009}.  There the limit shape of the boxed partitions 
is identified as the portion of the curve $\{ e^{-x} + e^{-y} = 1 \}$,
which is the limit shape of unrestricted partitions.  The portion 
is determined implicitly by the restriction that the endpoints 
of the curve are the opposite corners of a $1 \times A$-proportional 
rectangle and that the area under the curve has the desired proportion, 
that is $B/A$ of the total rectangular area.  To see that this 
matches~\eqref{eq:lim_curve} we can calculate the given portion explicitly.

Let $x=s_1,s_2$ be the starting and ending points of the bounding 
rectangle.  The side ratio and the area requirement are respectively
equivalent to
\begin{align*}
\frac{\log(1-e^{s_1})-\log(1-e^{-s_2})}{s_2-s_1} & = A   \\
\mbox{and} \hspace{4in} & \hspace{4in} \\
\int_{s_1}^{s_2} - \log(1-e^{-t})dt + (s_2-s_1) \log(1-e^{-s_2}) & = 
   B (s_2-s_1)^2  
\end{align*}
which simplify to
\begin{eqnarray}
A & = & \frac{1}{s_2-s_1} \log\left(\frac{e^{s_2}-1}{e^{s_2} 
   - e^{s_2-s_1}}\right) \label{eq:compare A} \, , \\
B & = & \frac{-\dilog(1-e^{-s_2}) + \dilog(1-e^{-s_1}) 
   + (s_2-s_1)\log(1-e^{-s_2})}{(s_2-s_1)^2} \, . \label{eq:compare B}
\end{eqnarray}

Comparing these equations with equations~\eqref{eq:integrals A} 
and~\eqref{eq:integrals B} 
it is immediate that the solutions are given by $s_1=c$ and $s_2=c+d$. 
Finally, to match the curve in the second line of equation~\eqref{eq:lim_curve} 
we need the coordinate transform from the curve $\gamma$ in the segment 
$x = [c,c+d]$ given by
$$x \to x_1 = \frac{(x-c)}{d}, \qquad y \to y_1-A =  
   \frac{y+\log(1-e^{-c})}{d}$$
whence $x = d x_1 + c$ and $y = -d (A-y_1) - \log(1-e^{-c})$ 
and the curves match.

\section{Existence and Uniqueness of \texorpdfstring{$c,d$}{c,d}}
\label{sec:constants}

We now show that for any $A \geq 2B > 0$ there exists unique positive constants $c$ and $d$ satisfying Equations~\eqref{eq:integrals A} and~\eqref{eq:integrals B}. If $A=B/2$ then $d=0$ and $c$ can be determined uniquely, so we may assume $A > 2B > 0$. The following lemma will be used to show uniqueness.

\begin{lem} \label{lem:J}
Let $\psi$ denote the map taking the pair $(c,d)$ to $(A,B)$ defined by the two integrals in Equations~\eqref{eq:integrals A} and~\eqref{eq:integrals B}, and let $K$ be a compact subset of $\{ (x,y) : x > 2y > 0 \}$. The Jacobian matrix $J := D[\psi]$ is negative definite for all $(c,d) \in (0,\infty)^2$, and all entries of $\psi$ and $J$ (respectively $\psi^{-1}$ and $J^{-1}$) are Lipshitz continuous on $\psi^{-1}[K]$ (respectively $K$).
\end{lem}

\begin{proof}
Differentiating under the integral sign shows that the partial derivatives comprising the entries of $D[\psi]$ are given by
\begin{eqnarray*}
J_{A,c} & = & \int_0^1 \frac{- e^{-(c+dt)}}{(1 - e^{-(c+dt)})^2} \, dt \\
J_{A,d} & = & \int_0^1 \frac{- t \, e^{-(c+dt)}}{(1 - e^{-(c+dt)})^2} \, dt \\
J_{B,c} & = & \int_0^1 \frac{- t \, e^{-(c+dt)}}{(1 - e^{-(c+dt)})^2} \, dt \\
J_{B,d} & = & \int_0^1 \frac{- t^2 \, e^{-(c+dt)}}{(1 - e^{-(c+dt)})^2} 
   \, dt  \, ;
\end{eqnarray*}
note that each term is negative. 
Let $\rho$ denote the finite measure on $[0,1]$ with density $e^{-(c+dt)} / (1 - e^{-(c+dt)})^2$ and let $\E_\rho$ denote expectation with respect to $\rho$.  Then
\[ J_{A,c} =  \E_\rho [-1], \qquad J_{A,d} = J_{B,c} =  \E_\rho [-t], \qquad J_{B,d} = \E_\rho [-t^2], \]
and
\[ \det J =  \E_\rho [1] \cdot \E_\rho [t^2] - \left(\E_\rho [t]\right)^2 = \E_{\rho}[1]^2 \cdot \Var_{\sigma}[t], \]
where $\Var_{\sigma}[t]$ denotes the variance of $t$ with respect to the normalized measure $\sigma = \rho/\E_{\rho}[1]$. In particular, $\det J$ is positive, and bounded above and below when $c$ and $d$ are bounded away from 0, implying the stated results on Lipshitz continuity. As $J$ is real and symmetric, it has real eigenvalues. Since the trace of $J$ is negative while its determinant is positive, the eigenvalues of $J$ have negative sum and positive product, meaning both are strictly negative and $J$ is negative definite for any $c,d>0$.
\end{proof}

\begin{lem}
\label{lem:uniquecd}
For any $A>0$ and $B \in (0,A/2)$ there exist unique $c,d>0$ satisfying Equations~\eqref{eq:integrals A} and~\eqref{eq:integrals B}. Moreover, for a fixed $A$, when $B$ decreases from $A/2$ to $0$ then $d$ increases strictly from $0$ to $\infty$ and $c$ decreases strictly from $\log\left(\frac{A+1}{A}\right)$ to $1$. When $B>0$ is fixed and $A$ goes to $\infty$ then $c$ goes to $0$ and $d$ goes to the root of 
$$d^2 =B \left( d\log (1-e^{-d}) -\dilog(1-e^{-d}) \right).$$
\end{lem}

\begin{proof}
Solving Equation~\eqref{eq:integrals A} for $c$ (assuming $d\geq 0$) gives 
\[ c=\log \left( \frac{e^{(A+1)d}-1}{e^{(A+1)d}-e^d}\right). \]
Substituting this into Equation~\eqref{eq:integrals B} gives an explicit expression for $B$ in terms of $A$ and $d$, and shows that for fixed $A>0$ as $d$ goes from 0 to infinity $B$ goes from $A/2$ to 0. By continuity, this implies the existence of the desired $c$ and $d$. It also shows that,  for a fixed $A$, $c$ is a decreasing function of $d$ with the given maximal and minimal values as $d$ goes from $0$ to $\infty$.

To prove uniqueness, we note that for $\xx,\yy \in \mathbb{R}^2$ Stokes' theorem implies
\[ \psi(\yy) - \psi(\xx) = \int_0^1 D[\psi]\left( t\xx + (1-t)\yy \right) \cdot (\xx - \yy) \, dt  \]
so that
\[ (\xx-\yy)^T \cdot \left(\psi(\yy) - \psi(\xx)\right) = \int_0^1 \left[(\xx-\yy)^T \cdot D[\psi]\left( t\xx + (1-t)\yy \right) \cdot (\xx - \yy)\right] dt.  \]
When $\xx \neq \yy$, negative-definiteness of $D[\psi]$ implies that the last integrand is strictly negative on $[0,1]$, and $\psi(\yy) \neq \psi(\xx)$. Thus, distinct values of $c$ and $d$ give distinct values of $A$ and $B$.

To see the monotonicity, let $A$ be fixed and let $F_B(d)=B$ be the equation obtained after substituting $c=c(A,d)$ above in Equation~\eqref{eq:integrals B}, i.e. $F_B(d) = \psi_2(c(A,d),d)$. Then $d$ is a decreasing function of $B$ and vice versa since
\[ \frac{\partial F_B(d)}{\partial d} = \frac{J_{B,d}J_{A,c} - J_{A,d}J_{B,c}}{J_{A,c}} = \frac{\det D[\psi]}{J_{A,c}} <0\, . \]

For the last part, the explicit formula for $c$ in terms of $A$ and $d$ shows that $c\to 0$. Substitution in Equation~\eqref{eq:integrals B} gives the desired equation.

\end{proof}

\section{Proof of Theorem~\protect{\ref{th:main}} from the discretized result}
\label{sec:implication}

Here we show how $c_m$ and $d_m$ from the discretized result are related to $c,d$ defined independently of $m$. The proof below also shows that $c_m$ and $d_m$ exist and are unique.

The Euler-MacLaurin summation formula~\cite[Section 3.6]{Bruijn1958} gives an expansion
\begin{align}
\frac{L_m}{m} &= \int_0^1 \log(1-e^{-c_m-d_mt})\, dt + \frac{\log(1-e^{-c_m})+\log(1-e^{-c_m-d_m})}{2m} + O(m^{-2}) \notag \\[+2mm]
&=  \frac{\dilog(1 - e^{-c_m-d_m}) - \dilog(1-e^{-c_m})}{d_m} + \frac{\log(1-e^{-c_m})+\log(1-e^{-c_m-d_m})}{2m} + O(m^{-2}) \label{eq:Lm1}
\end{align}
of the sum $L_m$ in terms of $c_m$ and $d_m$. Assume that there is an asymptotic expansion
\begin{eqnarray} 
c_m & = & c + u m^{-1} + O(m^{-2}) \label{eq:c+} \\
d_m & = & d + v m^{-1} + O(m^{-2}) \label{eq:d+}
\end{eqnarray}
as $m\rightarrow\infty$, where $u$ and $v$ are constants depending only on $A$ and $B$. Under such an assumption, substitution of Equations~\eqref{eq:c+} and~\eqref{eq:d+} into Equation~\eqref{eq:Lm1} implies
\begin{align}
\label{eq:L+}
\frac{L_m}{m} &=  \frac{\dilog(1 - e^{-c-d}) - \dilog(1-e^{-c})}{d} + \frac{u A + v B}{m} + O(m^{-2}) \notag \\
&= \log(1-e^{-c-d}) - dB + \frac{u A + v B}{m} + O(m^{-2}).
\end{align}
Substituting Equations~\eqref{eq:c+}--\eqref{eq:L+} into Equation~\eqref{eq:N'} of Theorem~\ref{th:1'} and taking the limit as $m \rightarrow \infty$ then gives Theorem~\ref{th:main}, as
\[ 
\Delta_m\rightarrow \left(\int_0^1 \frac{e^{-c-dt}}{(1-e^{-c-dt})^2}\,dt\right)\left(\int_0^1 \frac{t^2e^{-c-dt}}{(1-e^{-c-dt})^2}\,dt\right) - \left(\int_0^1 \frac{te^{-c-dt}}{(1-e^{-c-dt})^2}\,dt\right)^2 = \Delta. 
\]

It remains to show the expansions in Equations~\eqref{eq:c+} and~\eqref{eq:d+}. For $x,y>0$, define 
\begin{align*}
\Sbar_m(x,y) & := \frac{1}{m} \sum_{j=0}^m \frac{1}{1 - e^{-(x+yj/m)}} - 1 \, , \\
\Tbar_m(x,y) & := \frac{1}{m} \sum_{j=0}^m \frac{j/m}{1 - e^{-(x+yj/m)}} - \frac{1}{2}  \, .
\end{align*}
Another application of the Euler-MacLaurin summation formula implies
\begin{align}
\Sbar_m(c,d) &=  A + A_1 (c,d) m^{-1} + O(m^{-2}) \, , \label{eq:A1} \\
\Tbar_m(c,d) &=  B + B_1 (c,d) m^{-1} + O(m^{-2}) \, , \label{eq:B1}
\end{align}
with
\[ A_1 = \frac{1}{2}\left(\frac{1}{1-e^{-c}} + \frac{1}{1-e^{-c-d}}\right) \qquad \text{and} \qquad B_1 = \frac{1}{2(1-e^{-c-d})}. \]
Let $\J$ denote the Jacobian $D[\psi]$ of the map $\psi$, introduced in Lemma~\ref{lem:J}, with respect to $c$ and $d$, and let
\[ (c_m' , d_m') = (c,d) - m^{-1} \J^{-1} \cdot (A_1-1 , B_1-1/2)^T.\]
A Taylor expansion around the point $(c,d)$ gives
\begin{align*}
\begin{pmatrix}
\Sbar_m(c_m',d_m') \\
\Tbar_m(c_m',d_m')
\end{pmatrix}
&= \begin{pmatrix}
\Sbar_m(c,d) \\
\Tbar_m(c,d)
\end{pmatrix} - \left(\J + O\left(m^{-1}\right) \right) \cdot \left(m^{-1} \J^{-1} \begin{pmatrix}
A_1 \\
B_1
\end{pmatrix}
\right) + O(m^{-2}) \\
&= 
\begin{pmatrix}
A-1/m \\
B-1/2m
\end{pmatrix} + O\left(m^{-2}\right) \\
&= 
\begin{pmatrix}
\Sbar_m(c_m,d_m) \\
\Tbar_m(c_m,d_m)
\end{pmatrix} + O\left(m^{-2}\right),
\end{align*} 
where Equations~\eqref{eq:A1} and~\eqref{eq:B1} were used to approximate the Jacobian of $\psi_m: (x,y) \mapsto \left(\Sbar_m(x,y),\Tbar_m(x,y)\right)$ with respect to $x$ and $y$.

The map $\psi_m$ is Lipschitz for a similar reason as its continuous analogue. Namely, consider the partial derivatives
\begin{eqnarray*}
J_{S,x} & = & \frac{1}{m}\sum_{j=0}^m \  -\frac{e^{-x-yj/m}}{(1-e^{-x-yj/m})^2}\, \\
J_{S,y} & = & \frac{1}{m^2}\sum_{j=0}^m \  -\frac{j\ e^{-x-yj/m}}{(1-e^{-x-yj/m})^2}\, \\
J_{T,x} & = & \frac{1}{m^2}\sum_{j=0}^m \  -\frac{j\ e^{-x-yj/m}}{(1-e^{-x-yj/m})^2}\, \\
J_{S,y} & = & \frac{1}{m^3}\sum_{j=0}^m \  -\frac{j^2\ e^{-x-yj/m}}{(1-e^{-x-yj/m})^2}.
\end{eqnarray*}
Let $\rho_m$ be a discrete finite measure on $R_m:=\{0,1/m,2/m,\ldots,1\}$ with density $e^{-x-yt}/(1-e^{-x-yt})$ for $t\in R_m$ and 0 otherwise, and let $\E_{\rho_m}$ be the expectation with respect to $\rho_m$. Then 
\[ J_{S,x} = \E_{\rho_m}[-1] \, , \qquad J_{T,x}=J_{S,y} = \E_{\rho_m}[-t] \, , \qquad J_{T,y} =\E_{\rho_m}[-t^2]\]
and
\[ \det D[\psi_m] = \E_{\rho_m}[1]\E_{\rho_m}[t^2] - \E_{\rho_m}[t]^2 = \E_{\rho_m}[1]^2 \Var_{\sigma_m} [t]\, ,\]
where $\sigma_m$ is the probability function $\rho_m/\E_{\rho_m}[1]$. For any fixed $m$ and $(x,y)$ in a compact neighborhood of $(A,B)$, both the variance and the expectation are finite and bounded away from 0, as is the Jacobian determinant. Moreover, the trace ${\rm Tr} D[\psi] = -\E_{\rho_m}[1+t^2]$ is  bounded away from 0 and infinity, so the Jacobian is negative definite with locally bounded eigenvalues, and hence $\psi_m$ is locally Lipschitz.
 Since the norm of the Jacobian is bounded away from 0 and infinity, we have that the inverse map $\psi_m^{-1}$ is also locally Lipschitz in a neighborhood of $\psi^{-1}(A,B)$.  Moreover, similarly to proof of existence and uniqueness of $c$ and $d$ in Section~\ref{sec:constants}, we have that there indeed are $c_m$ and $d_m$ as unique solutions of Equations~\eqref{eq:mu_m} and~\eqref{eq:nu_m} since the Jacobian is negative semi-definite.  

The trapezoid formula implies $\left|J_{S,c}-J_{A,c}\right| = O(m^{-1})$, and similar bounds for the other differences of partial derivatives in the continuous and discrete settings.  Hence, the bounds for the norms and eigenvalues of $D[\psi_m]$ are within $O(m^{-1})$ of the ones for $D[\psi]$, and $\psi_m$ (and its inverse) is Lipschitz with a constant independent of $m$. Thus,
\[ O(m^{-2})=\|\psi_m(c'_m,d'_m) -\psi_m(c_m,d_m)\|\geq C^{-1} \| (c_m'-c_m,d'_m-d_m)\| \]
for some constant $C$, so that the expansions~\eqref{eq:c+} and~\eqref{eq:d+} hold. 
$\Cox$

\section{Proof of Theorem~\protect{\ref{th:difference}}}
\label{sec:proof}
We will prove Theorem~\ref{th:difference} from Equation~\eqref{eq:P_m} and Corollary~\ref{cor:lclt_error}. Let $p_m(\ell,n) = \P_m \left [ (S_m,T_m) = (\ell , n) \right ]$ and let
\begin{align}\label{eq:discrete_funcs}
L_m(x,y) &:= \sum_{j=0}^m \log(1-e^{-x-yj/m}) \ ,\\
A_m(x,y) &:= \sum_{j=0}^m \frac{1}{1-e^{-x-yj/m}} - (m+1) \ ,\\
B_m(x,y) &:= \sum_{j=0}^m \frac{j/m}{1-e^{-x-yj/m}} - \frac{m+1}{2} \ .
\end{align}
Then $c_m$ and $d_m$ are the solutions to 
\[ A_m(c_m,d_m) = \ell= Am  \, , \qquad B_m(c_m,d_m) = n/m = Bm \, . \]
Let $c'_m,d'_m$ be the solutions to $A_m(c'_m,d'_m)=\ell$ and $B_m(c'_m,d'_m)=(n+1)/m$, and let $\Delta x = c'_m-c_m=O(m^{-2})$ and $\Delta y = d'_m-d_m=O(m^{-2})$ by the Lipschitz properties proven in Section~\ref{sec:constants}. Observe that 
\begin{equation}\label{eq:derivatives}
 \frac{\partial L_m(x,y) }{\partial x} = A_m(x,y) \quad  \text{and} \quad \frac{\partial L_m(x,y)}{\partial y} = B_m(x,y).
\end{equation}
Using the Taylor expansion for $L_m(c'_m,d'_m)$ around $(c_m,d_m)$ and the $L_m$ partial derivatives from Equation~\eqref{eq:derivatives},
\[ -L_m(c_m', d_m') = -L_m(c_m+\Delta x, d_m +\Delta y) = -L_m(c_m,d_m) - \Delta x \, A_m(c_m,d_m) - \Delta y \, B_m(c_m,d_y) + O(m^{-3}), \]
so that
\begin{align*}
-L_m(c_m', d_m') + (c_m+\Delta x) \ell + (d_m + \Delta y) (n + 1) \, m^{-1} =-L_m(c_m,d_m) + c_m\ell  + d_m(n+1) \, m^{-1}  + O(m^{-3}).
\end{align*}
To lighten notation, we now write $L_m := L_m(c_m,d_m)$ and $L_m' := L_m(c_m',d_m')$. Then
\begin{align}
N_{n+1}(\ell,m) - N_n(\ell,m) 
&= p_m(\ell,n+1)\exp \left[ -L_m' + c_m' \ell + \frac{d_m'}{m}(n+1) \right] - p_m(\ell,n)\exp \left[ -L_m + c_m \ell + \frac{d_m}{m}n \right] \notag \\[+2mm]
&= p_m(\ell,n)\exp \left[ -L_m + c_m \ell + \frac{d_m}{m}n \right]\left[e^{d_m/m}-1\right] \label{eq:diffE1} \\[+1mm]
& \quad + \, \left[p_m(\ell,n+1)-p_m(\ell,n)\right]\exp \left[ -L_m + c_m \ell + \frac{d_m}{m}(n+1) \right] \label{eq:diffE2} \\[+1mm]
& \quad + \, p_m(\ell,n+1)\left( e^{-L_m' + c_m' \ell +d_m'(n+1)/m} -  e^{-L_m + c_m \ell + d_m(n+1)/m} \right). \label{eq:diffE3}
\end{align}
We now bound each of these summands.
\begin{itemize}
	\item Since $d_m = d + O(m^{-1})$, Equation~\eqref{eq:P_m} implies that the quantity on line~\eqref{eq:diffE1} equals 
	\[ N_n(\ell,m)\left(\frac{d}{m} + O(m^{-2}) \right) \]
	as long as $d \notin O(m^{-1})$. This holds when $|A-B/2| \notin O(m^{-1})$ as $d=0$ when $A=B/2$ and the map taking $(A,B)$ to $(c,d)$ is Lipschitz. 

	\item By Corollary~\ref{cor:lclt_error}, 
	\begin{align*} 
	\left[p_m(\ell,n+1)-p_m(\ell,n)\right] &\leq \left|\mN_m(\ell,n+1)-\mN_m(\ell,n)\right| + O(m^{-4}) \\
	&=  O\left(m^{-2} \cdot \left|1 - e^{\frac{1}{2} Q_m (0,1)} \right| \right) + O(m^{-4}) \\
	&= O(m^{-4}),
	\end{align*}
	where $Q_m$ is the inverse of the covariance matrix of $(S_m,T_m)$. Thus, the quantity on line~\eqref{eq:diffE2} is $O(m^{-4} \cdot m^2 N_n(\ell,m)) = O(m^{-2}N_n(\ell,m))$.

	\item Let 
	\[ \psi_m := \exp\bigg[ -L_m' + c_m' \ell +d_m'(n+1) \, m^{-1} - \left(-L_m + c_m \ell +d_m(n+1) \, m^{-1}\right) \bigg] - 1 = O(m^{-3}). \]
	As $p_m(\ell,n+1) = p_m(\ell,n) + O(m^{-4})$, it follows that the quantity on line~\eqref{eq:diffE3} is
	\begin{align*}
	p_m(\ell,n+1) \, e^{-L_m + c_m \ell + d_m(n+1)/m} \, \psi_m &= N_n(\ell,m) \, \psi_m \, e^{d_m/m} + O(m^{-4} \, e^{d_m/m} \, e^{-L_m + c_m \ell + d_m n/m} \, \psi_m) \\
	&= O(m^{-3} N_n(\ell,m)).
	\end{align*}
\end{itemize} 

Putting everything together, 
\[ N_{n+1}(\ell,m) - N_n(\ell,m) = N_n(\ell,m)\left(\frac{d}{m} + O(m^{-2}) \right), \]
as desired.
$\Cox$

\section*{Appendix: Proof of the Local Central Limit Theorem}

Throughout this section, $1/2 \geq \delta > 0$ is fixed and $\{ p_j : 0 \leq j \leq m \}$ are arbitrary numbers in $[\delta , 1-\delta]$. The variables $\{ X_j \}$ and $(S_m, T_m)$ are as in Lemma~\ref{lem:lclt}; we drop the index $m$ on the remaining quantities $\alpha_m, \beta_m, \gamma_m, \Delta_m, \mu_m, \nu_m, p_m(a,b)$ and the matrices $M_m$ and $Q_m$. Recall the quadratic form notation $M(s,t):= [s\, , \,\, t] \, M\, [s\,,\,\,t]^T$.

\begin{lem} \label{lem:delta}
The constants $\alpha, \beta, \gamma$ and $\Delta$ are bounded below and above by positive constants depending only on $\delta$.
\end{lem}

\noindent{\sc Proof:}
Upper and lower bounds on $\alpha, \beta$ and $\gamma$ are elementary:
$\disp \alpha \in \left [ \frac{\delta}{(1-\delta)^2}, \frac{(1-\delta)}{\delta^2} \right ] , \beta \in \left [ \frac{\delta}{2(1-\delta)^2}, \frac{(1-\delta)}{2\delta^2} \right ]$
and
$\disp \gamma \in \left [ \frac{\delta}{3(1-\delta)^2}, \frac{(1-\delta)}{3\delta^2} \right ]$.
The upper bound on $\Delta$ follows from these.

For the lower bound on $\Delta$, let 
$\disp {\Mt} = \left ( \begin{array}{cc} \alpha_n & \beta_n \\ \beta_n & \gamma_n \end{array} \right )$ 
denote $M$ without the factors of $m$. We show $\Delta$ is bounded from below by the positive constant $(4 - \sqrt{13}) \delta / 6$. A lower bound for the determinant $\Delta$ of $\Mt$ is $|\lambda|^2$ where $\lambda$ is the least modulus eigenvalue of $\Mt$; note that $|\lambda|^2 = \inf_\theta \Mt (\cos \theta , \sin \theta)$.  We compute
\begin{align*}
\Mt (\cos \theta , \sin \theta) & = m^{-1} \E \left (\cos \theta S + m^{-1} \sin \theta T \right )^2 \\
& \geq \delta m^{-1} \sum_{k=0}^m \left ( \cos \theta + \frac{k}{m} \sin \theta \right )^2 \\
& > \delta \cdot \left ( \cos^2 \theta + \cos \theta \sin \theta + \frac{1}{3} \sin^2 \theta \right ) \, .
\end{align*}
This is at least $\disp \frac{4 - \sqrt{13}}{6} \delta$ for all $\theta$, proving the lemma.
$\Cox$

\begin{lem} \label{lem:unif}
Let $X_p$ denote a reduced geometric with parameter $p$. For every $\delta \in (0,1/2)$ there is a $K$ such that simultaneously for all $p \in [\delta , 1-\delta]$,
\[ \left | \log \E \exp(i \lambda X_p) - \left ( i \frac{q}{p} \lambda - \frac{q^2}{2 p^2} \lambda^2 \right ) \right | \leq K \lambda^3 \, . \]
\end{lem}

\noindent{\sc Proof:}
For fixed $p$ this is Taylor's remainder theorem together with the fact that the characteristic function $\phi_p (\lambda)$ of $X_p$ is thrice differentiable.  The constant $K(p)$ one obtains this way is continuous in $p$ on the interval $(0,1)$, therefore bounded on any compact sub-interval.
$\Cox$

\noindent{\sc Proof of the LCLT:}
The proof of Lemma~\ref{lem:lclt} comes from expressing the probability as an integral of the characteristic function, via the inversion formula, and then estimating the integrand in various regions.

Let $\phi(s,t) := \E e^{i (s S + t T)}$ denote the characteristic function of $(S , T)$.  Centering the variables at their means, denote $\ov{S} := S - \mu$, $\ov{T} := T-\nu$, and $\ov{\phi}(s,t) := \E e^{i(s\ov{S} +t \ov{T})}$ so that $\phi(s,t) = \ov{\phi}(s,t) e^{is\mu + it\nu}$.  Then
\begin{align} 
p(a,b) & = \frac{1}{(2 \pi)^2} \int_{-\pi}^\pi \int_{-\pi}^\pi e^{-i s a - i t b} \phi (s,t) \, ds \, dt \nonumber \\
& = \frac{1}{(2 \pi)^2} \int_{-\pi}^\pi \int_{-\pi}^\pi e^{-i s (a-\mu) - i t (b-\mu)} \phihat (s,t) \, ds \, dt \, .  \label{eq:inversion}
\end{align}
Following the proof of the univariate LCLT for IID variables found in~\cite{Durrett2010}, we observe that
\begin{equation} \label{eq:normal}
\frac{1}{2 \pi (\det M)^{1/2}} e^{-\frac{1}{2} Q (a - \mu , b - \nu)}
   = \frac{1}{(2 \pi)^2} \int_{-\infty}^{\infty} \int_{-\infty}^{\infty} 
   e^{-i s (a-u) - i t (b-v)} 
   \exp \left ( - \frac{1}{2} M(s,t) \right ) \, ds \, dt \, .
\end{equation}
Hence, comparing this to~\eqref{eq:inversion} and observing that $e^{-is(a-\mu) - it(b-\nu)}$ has unit modulus, the absolute difference between $p(a,b)$ and the left-hand side of~\eqref{eq:normal} is bounded above by
\begin{equation} \label{eq:diff}
   \frac{1}{(2 \pi)^2}\int_{-\infty}^{\infty} \int_{-\infty}^{\infty} 
   \left | \one_{(s,t) \in [-\pi,\pi]^2} \ov{\phi} (s,t) 
   - e^{-(1/2) M(s,t)} \right | ds \, dt \, .
\end{equation}

Fix positive constants $L$ and $\ee$ to be specified later and decompose the region $\region := [-\pi,\pi]^2$ as the disjoint union $\regiona + \regionb + \regionc$, where
\begin{align*}
\regiona & = [-L m^{-1/2} , L m^{-1/2}] \times [-L m^{-3/2} , L m^{-3/2}] \\
\regionb & = [-\ee , \ee] \times [-\ee m^{-1} , \ee m^{-1}] 
   \;\; \setminus \; \regiona \\
\regionc & = \region \setminus (\regiona \cup \regionb) \, ;
\end{align*}
see Figure~\ref{fig:rectangles} for details.

\begin{figure}
\centering
\includegraphics[width=0.3\linewidth]{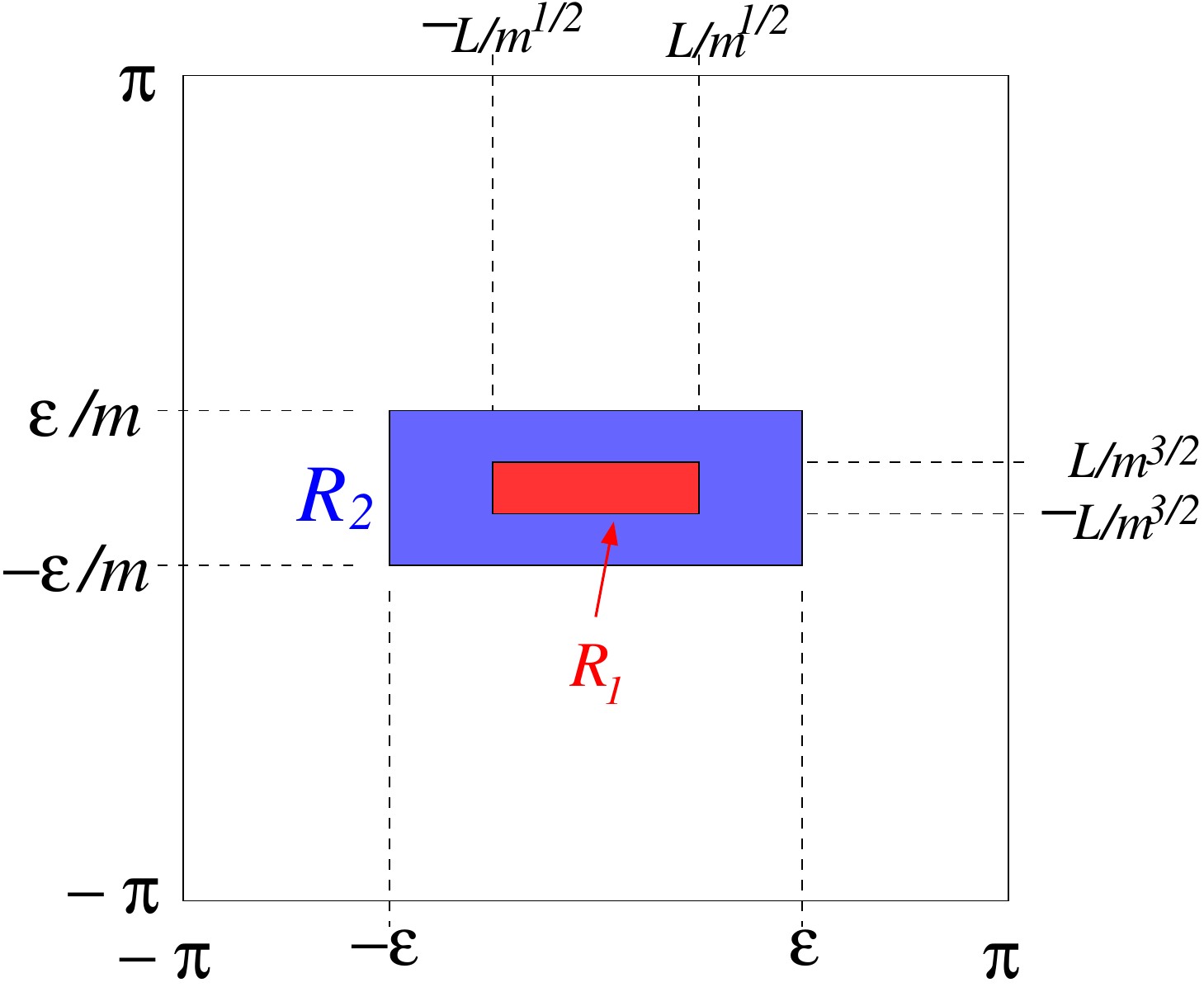}
\caption{The regions $\regiona \subseteq \regionb \subseteq \region$ in the proof of the LCLT.}
\label{fig:rectangles}
\end{figure}

As $\int_{\regionb^c} e^{- (1/2) M(s,t)} \, ds \, dt$ decays exponentially with $m$, it suffices to obtain the following estimates
\begin{align}
\int_{\regiona} \left | \phihat(s,t) - e^{-(1/2) M(s,t)} \right | \, ds \, dt
   & = O\left ( m^{-5/2} \right ) \label{eq:i}\\
\int_{\regionb} \left |  \phihat(s,t) - e^{-(1/2) M(s,t)} \right | \, ds \, dt
    & = O(m^{-5/2}) \label{eq:ii} \\
\int_{\regionc} \left | \phihat(s,t) \right | \, ds \, dt
   & = o \left ( m^{-3} \right ). \label{eq:iii}
\end{align}
By independence of $\{ X_j \}$, 
\[ \log \phihat (s,t) = \sum_{j=0}^m \log \E e^{i (s+jt) (X_j - \mu_j)} \, .\]
Using Lemma~\ref{lem:unif} with $p = p_j$ gives
\[ \left | \log \E e^{i (s+jt) (X_j - q_j/p_j)} + \frac{q_j}{2 p_j^2} (s+jt)^2 \right | \leq K |s+jt|^3 \, .\]
The sum of $(q_j / p_j^2) (s+jt)^2$ is $M(s,t)$, therefore summing the previous inequalities over $j$ gives
\begin{equation} 
\label{eq:R1}
\left | \log \phihat(s,t) + \frac{1}{2} M(s,t) \right | \leq K \sum_{j=0}^m |s + jt|^3 \, .
\end{equation}
On $\regiona$ we have the upper bound $|s + jt| \leq |s| + m |t| \leq 2 L m^{-1/2}$.  Thus,
\[ \sum_{j=0}^m |s+jt|^3 \leq (m+1) (8 L^3) m^{-3/2} = O \left ( m^{-1/2} \right ) \, .\]
Plugging this into~\eqref{eq:R1} and exponentiating shows that the left hand side of~\eqref{eq:i} is at most $|\regiona| \cdot O(m^{-1/2}) = O(m^{-5/2})$.

To bound the integral on $\regionb$, we define the sub-regions
\[ S_k := \left\{(x,y) : k \leq \max\left(m^{1/2}|x|,m^{3/2}|y|\right) \leq k+1 \right\}.\]
As the area of $S_k$ is $(8k+4)m^{-2}$,
\begin{align}
\int_{\regionb} \left |  \phihat(s,t) - e^{-(1/2) M(s,t)} \right | \, ds \, dt
&\leq \sum_{k=L}^{\lceil \, \epsilon \sqrt{m} \, \rceil} \int_{S_k}\left|\phihat(s,t)-e^{-M(s,t)/2}\right|dsdt \notag \\
&\leq m^{-2} \sum_{k=L}^{\lceil \, \epsilon \sqrt{m} \, \rceil} (8k+4) \max_{(s,t) \in S_k} \left|\phihat(s,t)-e^{-M(s,t)/2}\right| \,. \label{eq:R2m}
\end{align}
We break this last sum into two parts, and bound each part.  For $(s,t) \in \regionb$, we have $|s + j t| \leq |s| + m |t| \leq 2 \ee$ so that
\[ \sum_{j=0}^m |s + jt|^3 \leq 2 \ee \sum_{j=0}^m (|s| + j|t|)^2\leq (2 \ee \Delta^{-1}) M(|s|,|t|).\]  
Comparing this to~\eqref{eq:R1} shows we may choose $\ee$ small enough to guarantee that 
\[\left | \log \phihat(s,t) + \frac{1}{2} M(s,t) \right | \leq \frac{1}{4} M(|s|,|t|) \, , \]
so $|\phihat(s,t)| \leq e^{-(1/4) M(s,t)}$. Lemma~\ref{lem:delta} shows there is a positive constant $c$ such that the minimum value of $M(s,t)$ on $S_k$ is at least $ck^2$. Thus, for $(s,t) \in S_k$,
\[ \left|\phihat(s,t)-e^{-M(s,t)/2}\right| \leq \left|e^{-M(s,t)/4}\right| +  \left|e^{-M(s,t)/2}\right| \leq 2e^{-ck^2}.\]
If $r_m := \left\lceil \sqrt{(\log m)/c} \, \right\rceil$ then
\begin{align} 
\sum_{k=r_m}^{\infty} (8k+4)(k+1) \max_{(s,t) \in S_k} \left|\phihat(s,t)-e^{-M(s,t)/2}\right| 
&\leq  2\sum_{k=r_m}^{\infty} (8k+4)(k+1) e^{-ck^2} \notag \\
&= O(m^{-1} \, \text{polylog}(m)) \notag \\
&= O(m^{-1/2}),  \label{eq:R2m2}
\end{align}
where $\text{polylog}(m)$ denotes a quantity growing as an integer power of $\log m$. Furthermore, for $(s,t) \in S_k$ there exist constants $C$ and $C'$ such that
\[ \left|\log \phihat(s,t) + M(s,t)/2 \right| \leq C \sum_{j=0}^m \left|s+jt\right|^3 \leq C \left(2(k+1)m^{-1/2}\right)^3(m+1) = C'k^3m^{-1/2}.\]
This implies the existence of a constant $K>0$ such that for $0 \leq k \leq r_m$ and $(s,t) \in S_k$, 
\begin{align*}
\left|\phihat(s,t)-e^{-M(s,t)/2}\right|  
&= \left|e^{-M(s,t)/2}\right|\left|1-e^{\log \phihat(s,t)+ M(s,t)/2}\right|  \\
&\leq K e^{-ck^2} k^3m^{-1/2}.
\end{align*}
Thus,
\begin{align} 
\sum_{k=L}^{r_m} (8k+4)(k+1) \max_{(s,t) \in S_k} \left|\phihat(s,t)-e^{-M(s,t)/2}\right| 
&\leq  K m^{-1/2} \sum_{k=L}^{r_m} (8k+4)(k+1)k^3e^{-ck^2} \notag \\
&= O(m^{-1/2}). \label{eq:R2m3}
\end{align}
Combining~\eqref{eq:R2m}--\eqref{eq:R2m3} gives~\eqref{eq:ii}.

Finally, for~\eqref{eq:iii}, we claim there is a positive constant $c$ for which $|\phihat(s,t)| \leq e^{-cm}$ on $\regionc$.  To see this, observe (see~\cite[p. 144]{Durrett2010}) that for each $p$ there is an $\eta > 0$ such that $|\phi_p (\lambda)| < 1 - \eta$ on $[-\pi , \pi] \setminus [-\ee/2,\ee/2]$.  Again, by continuity, we may choose one such $\eta$ valid for all $p \in [\delta , 1-\delta]$.  It suffices to show that when either $|s|$ or $m |t|$ is at least $\ee$, then at least $m/3$ of the summands $\log \E e^{i(s+jt)(X_j - \mu_j)}$ have real part at most $-\eta$.  Suppose $s \geq \ee$ (the argument is the same for $s \leq -\ee$).  Interpreting $s+jt$ modulo $2\pi$ always to lie in $[-\pi,\pi]$, the number of $j \in [0,m]$ for which $s + jt \in [-\ee/2,\ee/2]$ is at most twice the number for which $s + jt \in [\ee/2,\ee]$, hence at most twice the number for which $s + jt \notin [-\ee/2,\ee/2]$; thus at least $m/3$ of the $m+1$ values of $s+jt$ lie outside $[-\ee/2,\ee/2]$ and these have real part of $\log \E e^{i(s+jt)(X_j-\mu_j)} \leq -\eta$ by choice of $\eta$.  Lastly, if instead one assumes $\pi \geq t \geq \ee/m$, then at most half of the values of $s+jt$ modulo $2\pi$ can fall inside any interval of length $\ee/2$.  Choosing $\eta$ such that the real part of $\log \E e^{i(s+jt)(X_j - \mu_j)}$ is at most $-\eta$ outside of $[-\ee/4,\ee/4]$ finishes the proof of~\eqref{eq:iii} and the LCLT. 
$\Cox$

\noindent{\sc Proof of Corollary~\ref{cor:lclt_error}.}
In order to estimate the error terms in the approximation of $p(a,b)$ we will consider the partial differences and repeat the approximation arguments above. Changing $b$ to $b+1$ in Equations~\eqref{eq:inversion} and~\eqref{eq:normal} implies 
\begin{equation}  \bigg|p(a,b+1)-p(a,b) - \big(\mN(a,b+1)-\mN(a,b)\big) \bigg| = \int_{[-\pi,\pi]^2} \left|1-e^{-it}\right| \left|\phihat(s,t) - e^{-1/2M(s,t)}\right|dsdt. \label{eq:pdiff} \end{equation}

For $(s,t) \in \regionc$, the proof of the LCLT shows that the integral in Equation~\eqref{eq:pdiff} decays exponentially with $m$. As $\left|1-e^{-it}\right| = \sqrt{2-2\cos(t)} \leq |t| = O(m^{-3/2})$ for $(s,t) \in \regiona$, the proof of the LCLT shows that the integral in Equation~\eqref{eq:pdiff} grows as $O(m^{-3/2} \cdot m^{-5/2}) = O(m^{-4})$. Finally, since $\left|1-e^{-it}\right| \leq |t| \leq (k+1)m^{-3/2}$ for $(s,t) \in S_k$ following the proof of the LCLT shows that
\begin{align*}
\int_{\regionb} \left|1-e^{-it}\right| \left|\phihat(s,t) - e^{-1/2M(s,t)}\right|dsdt 
& \leq m^{-7/2} \sum_{k=L}^{\lceil \, \epsilon \sqrt{m} \, \rceil} (8k+4)(k+1) \max_{(s,t) \in S_k} \left|\phihat(s,t)-e^{-M(s,t)/2}\right| \\
& = O(m^{-4}).
\end{align*}
$\Cox$

\section*{Acknowledgments}

We are very thankful to an anonymous referee of an earlier version.
This referee pointed out a body of literature that we had missed, 
in which the statistical mechanical ideas drawn on in this paper 
are already present.

\bibliographystyle{amsalpha}
\bibliography{bibl}

\end{document}